\documentclass{article} 
\usepackage{amsthm}
\usepackage{amsfonts}
\usepackage{amssymb}
\usepackage{mathtools}
\usepackage{enumerate}
\usepackage{graphicx}
\usepackage{algorithmic}
\usepackage{algorithm}
\usepackage{tikz}
\usepackage{caption}
\usepackage{multirow}
\usepackage[margin=0.65in]{geometry}
\usepackage{booktabs}
\newtheorem{thm}{Theorem}
\newtheorem*{remark}{Remark}
\DeclareMathOperator*{\argmin}{\arg\!\min}

\def\e{\epsilon}
\def\s{\sigma}
\def\k{\kappa}
\newtheorem{prop}{Proposition}
\usepackage{authblk}
\title{Exponentiated Subgradient Algorithm for Online Optimization under the Random Permutation Model}
\author[1]{Reza Eghbali\thanks{eghbali@uw.edu}}
\author[2]{Jon Swenson}
\author[1]{Maryam Fazel}
\affil[1]{Department of Electrical Engineering, University of Washington, Seattle}
\affil[2]{Department of Mathematics, University of Washington, Seattle}
\newlength\yearposx

\begin{document}

%Title
\maketitle
\begin{abstract}
Online optimization problems arise in many resource allocation tasks, where the future demands for each resource and the associated utility functions change over time and are not known apriori, yet resources need to be allocated at every point in time despite the future uncertainty. 
In this paper, we consider online optimization problems with general concave utilities. We modify and extend an online optimization algorithm proposed by Devanur et al. for linear programming to this general setting. The model we use for the arrival of the utilities and demands is known as the random permutation model, where a fixed collection of utilities and demands are presented to the algorithm in random order. We prove that under this model the algorithm achieves a competitive ratio of $1-O(\epsilon)$ under a near-optimal assumption that the bid to budget ratio is $O \left(\frac{\epsilon^2}{\log\left({m}/{\epsilon}\right)}\right)$,  where 
$m$ is the number of resources, while enjoying a significantly lower computational cost than the optimal algorithm proposed by  Kesselheim et al. We draw a connection between the proposed algorithm and subgradient methods used in convex optimization. In addition, we present numerical experiments that demonstrate the performance and speed of this algorithm in comparison to existing algorithms.
\end{abstract}
\section{Introduction}
Online optimization provides a framework for real-time optimization problems that arise in applications such as Internet advertising, online resource allocation, and online auctions. In these problems, functions of new variables are added to the objective, and constraints are modified over time. An online optimization algorithm should assign a value to each newly introduced variable without any knowledge of future changes in the objective function. Linear programs have received considerable attention in the context of online optimization (\cite{agrawal2009dynamic}, \cite{feldman2010online}, \cite{devanur2011near},\cite{jaillet2012near},\cite{kesselheimRTV13}). In this paper, we study a more general online convex optimization problem that contains \emph{online linear programming} (online LP) as a special case.  The performance of an online algorithm for any instance of the problem can be evaluated by finding the ratio $P/P^*$, where $P$ is the objective value achieved by the algorithm and $P^*$ is the optimal value of the corresponding offline optimization problem. The competitive ratio of an online maximization algorithm under the worst-case model is the infimum of $P/P^*$ over all the possible instances. While for many special online problems algorithms have been proposed that achieve a non-trivial competitive ratio under the worst-case model, the competitive ratio of any online LP algorithm tends to zero as the number of functions grows \cite{babaioff2008online}. To be able to derive non-trivial bounds on the competitive ratio of an online algorithm for any problem that contains online LP, one often considers the competitive ratio under more restrictive models such as the random permutation model or the i.i.d. model. 
%In the random permutation model, for any instance of the problem, the performance of an algorithm is evaluated by permuting the instance uniformly at random and evaluating $E[P]/ P^*$ where $E[P]$ denotes the expected value of $P$.
%Then the competitive ratio under this model is the infimum of $E[P]/ P^*$ over all the possible instances. 
%\textcolor{blue}{This may be a little confusing: what do we mean by 'instance'? is an instance defined by the problem data regardless of order, or does it refer to the data in a particular order? the confusing arises from saying 'permuting the instance'. Can we say it more clearly?} 
%The random permutation model, although more restrictive than the worst-case model, is less restrictive than the i.i.d. model in which new functions and constraint coefficients are independently sampled from a single distribution. 
The results of this paper are presented for the random permutation model.
In section \ref{form}, we formulate the problem and formally define the competitive ratio of an online algorithm under various models. 
%in a unified framework. 

\subsection{Problem formulation}\label{form}
 Consider the problem

\begin{align}\tag{$\mathcal{P}0$}\label{vector_general_i}
\underset{x_1,\ldots,x_n\in \mathbb{R}^k}{\mbox{maximize}} & {\quad \sum_{t=1}^{n} {f_t(x_t)}+\psi\left(\sum_{t=1}^{n}{A_t x_t }\right)},
\end{align}

\noindent where for all $t \in [n] = \{1,2,\ldots,n\}$, $A_t \in \mathbb{R}_{+}^{m\times k}$, $f_t \in \mathcal{G}$,  with $\mathcal{G}$ a subset of proper concave functions mapping $\mathbb{R}^k$ to $\left[-\infty,\, +\infty\right)$. More technical assumptions on the functions in $\mathcal{G}$ will be specified in section \ref{general}. Here $\psi(u)=I_{u \leq b}$ for some $b \in \mathbb{R}^m$, where $I$ denotes the indicator function of a  set, defined as follows for any $U \subset \mathbb{R}^m$:

\begin{align*}
I_{u \in U} = \left\{\begin{array}{cc} 0 & u \in U,\\ -\infty &  u \notin U.\end{array}\right.
\end{align*}

Often in applications, the functions $f_t$ are utility functions, i.e., a decision $x_t$ provides the utility $f_t(x_t)$, and the function $\psi=I_{u \leq b}$  can be interpreted as imposing a total \emph{budget} given by $b\in \mathbb{R}^m$. In an online optimization problem, for all $t \in [n]$, a vector should be assigned to $x_t$ that can depend on all the information obtained until time $t$, i.e., $\left(f_s,A_{s}\right)$, $1\leq s \leq t$, but not on the future information $\left(f_s,A_{s}\right)$, $t+1\leq s \leq n$.  
A procedure that performs this task is called an \emph{online optimization algorithm}. The input to the algorithm, $\left(\left(f_1,A_{1}\right),\left(f_2,A_{2}\right),\ldots,\left(f_n,A_{n}\right)\right)$, is a random vector taking values in $\mathcal{H}$, where $\mathcal{H} \subset  \left(\mathcal{G}\times \mathbb{R}_{+}^{m\times k}\right)^n $.  To avoid issues with measurability, we assume that this random vector takes values in a countable subset of $\mathcal{H}$. Let $\mathcal{F}$ be the set of probability distributions corresponding to such random vectors. To evaluate the performance of an online optimization algorithm, one can define the \emph{competitive ratio}. Let ${x}_t$, $1\leq t \leq n$ be the solution given by an online algorithm and define $P = \sum_{t=1}^{n} {f_t(x_t)}+\psi\left(\sum_{t=1}^{n}{A_t x_t }\right)$.  Let $P^*$ be the optimal value of \eqref{vector_general_i}. Assume that $P^* \in (0 ,\,+\infty)$ for all the members of $\mathcal{H}$. The competitive ratio over $\mathcal{D} \subset \mathcal{F}$ is defined as:

\begin{align*}
C_{\mathcal{D}} = \inf\left\{\left.E\left[\frac{P}{P^*}\right]\right\lvert {\left(\left(f_1,A_{1}\right),\left(f_2,A_{2}\right),\ldots,\left(f_n,A_{n}\right)\right)}\sim g \in \mathcal{D} \right\}.
\end{align*}
%The choice of $\mathcal{D}$ determines the model under which the competitive ratio is defined. 
Three important choices for $\mathcal{D}$ considered in the literature for special cases of problem \eqref{vector_general_i} are as follows: 
\begin{enumerate}
\item Worst-case: $\mathcal{D}_1 = \mathcal{F}$. In this case, $C_{\mathcal{D}_1}$ simplifies to:

\begin{align*}
C_{\mathcal{D}_1} = \inf\left\{\left.\frac{P}{P^*}\right\lvert \left(\left(f_1,A_{1}\right),\left(f_2,A_{2}\right),\ldots,\left(f_n,A_{n}\right)\right) \in \mathcal{H} \right\}.
\end{align*}

Note that if the online algorithm assigns a random value to $P$ (the algorithm is randomized), $P$ should be replaced by $E[P]$.
\item Random permutation: $\mathcal{D}_2$ is the subset of $\mathcal{F}$ corresponding to exchangeable random vectors, i.e., distribution functions that are invariant under permutation. 
%of their arguments. 
Let $\Pi$ be the set of all permutations over $\left[n\right]$. Let $\sigma$ be a random permutation uniformly distributed on $\Pi$. The definition of $C_{\mathcal{D}_2}$ simplifies to:

\begin{align*}
C_{\mathcal{D}_2} = \inf\left\{\left.\frac{E \left[P\right]}{P^*}\right\lvert \left(\left(f_1,A_{1}\right),\left(f_2,A_{2}\right),\ldots,\left(f_n,A_{n}\right)\right) =  \left(Y_{\s(1)},Y_{\s(2)},\ldots,Y_{\s(n)}\right), \, \left(Y_1,Y_2\ldots,Y_n\right) \in \mathcal{H}  \right\}.
\end{align*}

\item i.i.d.: $\mathcal{D}_3$ contains all $g \in \mathcal{F}$ induced by random vectors $\left(\left(f_1,A_{1}\right),\left(f_2,A_{2}\right),\ldots,\left(f_n,A_{n}\right)\right)$ such that $\left(f_t,A_t\right)$, $1\leq t\leq n$ are independent and identically distributed.

% such that $g(Y_1,Y_2,\ldots,Y_n)=\prod_{t=1}^{n}{\hat{g}}(Y_t)$, where $\hat{g}$ is a probability distribution on a countable subset of $\mathcal{H}$. In other words, $\left(f_t,A_t\right)$, $1\leq t\leq n$ are independent identically distributed random variables.
%\textcolor{blue}{Let's simplify this and state it words, so we can get rid of the extra notation $\hat g$.} \textcolor{red}{I honestly don't know how to state this only in words and end up with something that is accurate and not vague.}
\end{enumerate}

Note that since $\mathcal{D}_3 \subset \mathcal{D}_2 \subset \mathcal{D}_1$,  $C_{\mathcal{D}_3}\geq C_{\mathcal{D}_2}\geq C_{\mathcal{D}_1}$.

\subsection{Related work}
 
Most literature in online optimization has focused on a special case of \eqref{vector_general_i} in which for all $t\in[n]$, the domain of $f_t$  is the simplex 
$\{x \in [0,1]^k \mid \mathbf{1}^{\top} x \leq 1\}$ and $f_t(x) = c_t^T x$ for $x \in {\rm dom} f_t$.  Following \cite{agrawal2009dynamic}, we refer to this problem as the online LP problem, expressed as

\begin{align}\label{LP}
\begin{array}{ll}
\mbox{maximize}&{    \sum_{t=1}^{n} c_t^T x_t}\\ 
\mbox{subject to}& \sum_{t=1}^{n}{A_t x_t } \leq b\\
& \mathbf{1}^{\top} x_t \leq 1, x_t \geq 0 \;\; t\in [n].
\end{array}
\end{align}
Note that $C_{\mathcal{D}_1}$ for any online LP algorithm goes to zero when $n$ grows as shown by an example in \cite{babaioff2008online}, while for the same example under the random permutation model the optimal algorithm achieves a competitive ratio that tends to $1/e$ as $n$ goes to infinity  \cite{babaioff2008online}. In light of this fact, much attention has been dedicated to the random permutation and i.i.d. models for online LP. Here we review the related results for online LP.  In section~\ref{Apps}, we briefly review the special cases and the applications of combinatorial online optimization problems, with convex relaxations that can be expressed as online linear programs.

An important parameter that appears in the analysis of many online LP algorithms is $\gamma = \max_{t,i,j}{\frac{(A_{t})_{i,j}}{b_i}}$ which is called the \emph{bid-to-budget} ratio. For an online algorithm to achieve a competitive ratio of $1 - \e$ under random permutation or i.i.d. models, it is necessary that $\gamma = O\left( \frac{\e^2}{\log{m}}\right)$  when $m \geq 2$ (\cite{agrawal2009dynamic}, \cite{devanur2011near}). Feldman et al.\ \cite{feldman2010online} presented an algorithm with $C_{\mathcal{D}_2} \geq 1 - \e$ when $\gamma = O\left( \frac{\e^3}{m\log (nk)} \right)$ and  $\gamma' = O\left (\frac{\e}{\log(nk)} \right)$,  where $\gamma' := \frac{\max_{t,i} c_{t,i}}{P^*}$.  Agrawal et al.\ \cite{agrawal2009dynamic} proposed an online LP algorithm called \emph{dynamic learning algorithm (DLA)} with $C_{\mathcal{D}_2} = 1-O(\e)$ when $\gamma =O\left( \frac{\e^2}{m \log\left(nk/\e\right)}\right)$. DLA solves $\log_2(1/\e)$ linear programs to estimate the optimal dual variables associated with the $m$ constraints.  The primal variable is then assigned using the estimated dual variables. The dependency that exists on $n$ in their results is introduced by a union bound over the space of dual variables. Molinaro and Ravi \cite{molinaro2013geometry} improved this union bound and proved that a modified version of DLA achieves a competitive ratio of $1-O\left(\e\right) $ when $\gamma =O\left( \frac{\e^2}{m^2 \log\left({m}/{\e}\right)}\right)$. Kesselheim et al.\ \cite{kesselheimRTV13} proposed an algorithm that achieves a competitive ratio of $1-\e$  when $\gamma = O\left( \frac{\e^2}{1 + \log{d}}\right)$. Here $d$ is the maximum number of non-zero elements of any column of $A_t$ for all $t \in [n]$. When $d = m$ and $m \geq 2$ their condition gives $\gamma =   O\left( \frac{\e^2}{\log{m}}\right)$, which matches the necessary condition on $\gamma$.  When $d = 1$, their condition transforms to $\gamma =   O\left( {\e^2}\right)$, which also matches the necessary condition given in \cite{kleinberg2005multiple}. We call this algorithm \emph{KRTV} (dubbed after the initials of the authors' last names). At each $t\in [n]$, KRTV chooses the value of $x_t$ by solving a linear program with the same number of linear constraints as the original problem but with $t k$ variables. Devanur et al.\ \cite{devanur2011near} proposed an online algorithm that achieves a competitive ratio of $1-O(\e)$ under a model they call the adversarial stochastic model when $\max\{\gamma,  \gamma'\} =O\left( \frac{\e^2}{\log\left(m/\e\right)}\right)$. The adversarial stochastic model is more general than the i.i.d. model. However, their analysis does not apply to the random permutation model.

After finishing this paper, we became aware of two very recent results developed independent of this work and at around the same time. Authors in \cite{agrawal2014fast} and \cite{gupta2014experts} have analyzed similar algorithms to the one proposed in \cite{devanur2011near}.  In \cite{gupta2014experts}, the authors have analyzed this algorithm for online LP and demonstrated a competitive ratio of $1-O(\e$) when $\max\{\gamma,  \gamma'\} =O\left( \frac{\e^2}{\log\left(m/\e\right)}\right)$. They also prove their analysis holds true for a class of more general online linear programs that may include inequalities of form $\sum_{t=1}^{n}{A_t x_t } \geq b$. In \cite{agrawal2014fast}, authors have provided a competitive difference analysis for their algorithm applied to a class of general convex programs while showing that a specialized version of their algorithm for problem \eqref{vector_general_i} achieves a competitive ratio of $1-\e$ when $\max\{\gamma,  \gamma'\} =O\left( \frac{\e^2}{\log\left(m\right)}\right)$.
\subsection{Our results}

In this paper, we propose a modified and extended version of the online LP algorithm in \cite{devanur2011near} that applies to the more general problem \eqref{vector_general_i}. We prove that this algorithm achieves a competitive ratio of $1-O(\e)$ under the random permutation model when $\max\{\gamma,  \gamma'\} = O \left(\frac{\e^2}{\log\left({m}/{\e}\right)}\right)$ (the generalization of definition of $\gamma$ for \eqref{vector_general_i} is given in section \ref{general}). This condition is the same as the condition given in \cite{devanur2011near} for the adversarial stochastic model. 

We interpret the proposed algorithm as a form of subgradient descent, and call it \emph{exponentiated subgradient algorithm (ESA)}. ESA solves only $\log_2(1/\e)$ optimization problems with number of variables $n k\e,2 n k \e, 4 n k \e, \ldots, n k /2$. This makes ESA computationally as efficient as DLA, while having a theoretical guarantee on its competitive ratio under the random permutation model that is close to the theoretical guarantee for KRTV.

\subsection{Applications and special cases }
\label{Apps}

Consider the (non-convex) problem where $f_t(x) = c_t^T x$  over a binary domain ${\rm dom} f_t = \Gamma := \{x \in \{0,1\}^k \mid \mathbf{1}^{\top} x \leq 1\}$, which we refer to as the online allocation problem. It contains many combinatorial online problems studied in the literature such as online bipartite matching \cite{karp1990optimal}, the AdWords problem \cite{mehta2007adwords}, and the multiple secretary problem \cite{kleinberg2005multiple}. The convex relaxation of the online allocation problem is the online LP problem.  Note that all the online LP algorithms discussed above produce an integer solution; therefore, the bounds on their competitive ratio holds for the online allocation problem as well.

In many applications of online problems, the index $t$ is associated with time. We present an example of the online allocation problem using the terminology of servers and clients. Many other online allocation problems are equivalent to this example. Consider a server that has to serve a number of clients. At time $t$, a client sends an order of form $\left(c_t,A_t\right) \in \mathbb{R}_+ \times \mathbb{R}_+^m$. $A_t$ determines the amount of resources needed for the fulfillment of that order and $c_t$ is the price that the client is willing to pay for the order. From the point of view of the server, $c_t$ can be viewed as the utility associated with that order. The server has to make an irrevocable decision by choosing $x_t \in \{0,1\}$. If $x_t = 1$, the order is fulfilled and the total utility is increased by $c_t $; otherwise, the order is rejected and the total utility does not change. The goal of the server is to maximize the total utility while not exceeding any limit on resources. In the case where $k > 1$, each column of $A_t$ is an option. The server should make a decision of which option, if any, to fulfill  by choosing $x_t \in \Gamma$.  In the LP relaxation, each utility function is linear over its domain. One can assume a scenario where the order can be partially fulfilled with a utility function associated with each order that follows a diminishing return. In this case, every order is of form $\left(f_t,A_t\right)$, where $f_t$, $t=1,\ldots,n$ are concave functions. Special cases of online optimization problems where there is a relationship between $c_t$ and $A_t$ have been extensively studied in the literature; here we briefly review the results for four important special cases. For a thorough review of the results on these examples, we refer the reader to \cite{mehta13}.

\begin{enumerate}
\item Online Bipartite Matching: Consider a bipartite graph $G = \left(\left(U,V\right),E\right)$ with $|U| = n$ and $|V| = m$, where $|U|$ denotes the cardinality of $U$. Let $A = [a_1, a_2,\ldots, a_n] \in \{0,1\}^{ m \times n }$ be the adjacency matrix for $G$. Online bipartite matching is a special case of online allocation problem with $k = m$, $b = \mathbf{1}$,  $A_t = \sum_{i=1}^{m}a_{t,i}e_i e_i^T$, and $c_t = a_t$ for all $t\in [n]$, where $\{e_1,e_2,\ldots,e_m\}$ is the standard basis for $\mathbb{R}^m$. Karp et al.\ \cite{karp1990optimal} proposed a randomized online bipartite matching algorithm, called the ranking algorithm, that achieves $1-1/e$ competitive ratio under the worst-case model, which they proved is the best achievable competitive ratio by any online matching algorithm under this model. Under the random permutation model, the ranking algorithm achieves a competitive ratio of at least $0.696  > 1 - 1/e$ \cite{mahdian2011online} and no algorithm can achieve a competitive ratio larger that $0.727$ \cite{karande2011online}. Several authors have proposed online algorithms and studied their competitive ratio under the i.i.d. model when the algorithm receives the distribution as an input (see, e.g., \cite{feldman2009online2}, \cite{manshadi2012online} and \cite{jaillet2013online}). A competitive ratio of 0.706 under this model is achieved by the algorithm proposed in \cite{jaillet2013online}.

%A more general version of this problem is online bipartite b-matching, where $\psi(u) = -I_{\{ u \leq b \mathbf{1} \}}$ for some $b \in \mathbb{N}$. Kalyanasundaram et al. proposed an algorithm for online bipartite k-matching that achieves a competitive ratio of $1-1/e$ as $b \rightarrow \infty$. 

\item AdWords: This problem is a generalization of online bipartite matching with $A \in \mathbb{R}_{+}^{ m \times n }$ and $b \in  \mathbb{R}_{++}^{ m  }$. The name of this problem comes from its application in online advertising by search engines. A realistic assumption in this problem is that $\gamma$ is bounded by a small number. That renders the results for competitive ratio of online LP algorithms under i.i.d. or random permutation models applicable to this problem. In fact, the online AdWords algorithm proposed by Devanur and Hayes \cite{devanur2009adwords} for random permutation model inspired the later results for online LP algorithms. Mehta et al.\ provided an algorithm with $C_{\mathcal{D}_1} \rightarrow 1-1/e$ as $\gamma \rightarrow 0$, which is also shown to be the best competitive ratio under this model \cite{mehta2007adwords}. Buchbinder et al.\ \cite{buchbinder2007online} proposed an algorithm for the worst-case model that explicitly utilizes the dual variables and achieves the optimal competitive ratio. This algorithm retains a vector of dual variables and chooses the primal variable $x_t$, according to the complementary slackness rule. The dual variables are updated after each assignment of $x_t$. Devanur and Jain \cite{devanur2012online} proposed an algorithm for a concave version of AdWords, which linearizes the problem at each step and utilizes a similar primal-dual approach.
% This problem is a special case of \eqref{vector_general_i} with  $f_t = 0 $ for all $t\in[n]$ and $\psi(u) = \sum_{i=1}^{m}\psi_i(u_i)$, where each $\psi_i$ is a concave function on $\mathbb{R}$. 
% Devanur and Hayes study an online AdWords algorithm in the random permutation model that has competitive ratio $1-O(\e)$ when $\gamma' = O\left( \frac{\e^3}{m^2\log n} \right)$,\cite{devanur2009adwords}

\item Multiple Secretary Problem: This problem has been proposed by Kleinberg \cite{kleinberg2005multiple} as a variation of the classical secretary problem. It can be viewed as a special case of an online allocation problem with $k=1$, $m=1$, $c_t \in \mathbb{R}_+$, $A_t = 1$  and $b \in \mathbb{N}$. The pessimistic example for the competitive ratio of online LP algorithms under the worst-case model falls under this problem with $b = 1$. Therefore, this problem and many of its generalization are stated and studied under the random permutation model. Kleinberg \cite{kleinberg2005multiple} proposed an algorithm with $C_{\mathcal{D}_2}  = 1-O\left(1/\sqrt{b}\right)$, which was shown to be order-wise optimal. In other words, no online algorithm can achieve $C_{\mathcal{D}_2} \geq 1- \e$ unless $\gamma = O\left(\e^2\right)$.  Bibaioff et al. \cite{babaioff2007knapsack} proposed a simple algorithm that achieves a competitive ratio of $\frac{1}{e}$ for any $b$ as $n$ tends to infinity. 

%Multiple generalization of this problem have been considered in the literature. 

%\begin{align*}
%f_t(x) = \left \{ \begin{array}{cc}\mathbf{1}^{\top} x  & if  x \in \Gamma \\ -\infty & otherwise \end{array}\right. ,
%\end{align*}

\item  Online Weighted Bipartite Matching: This problem is a generalization of online bipartite matching. It can also be viewed as a generalization of the multiple secretary problem. Consider a bipartite graph and its adjacency matrix $A \in  \{0,1\}^{ m \times n }$. Let $W = [w_1, w_2,\ldots, w_n] \in \mathbb{R}_{+}^{ m \times n }$ be the weight matrix associated with the edges in the graph. In this problem, $k = m$, $b  \in \mathbb{N}^m$,  $A_t = \sum_{i=1}^{m}a_{t,i}e_i e_i^T$ and $c_t = w_t$ for all $t\in [n]$. Similar to the AdWords problem when $\gamma$ is small, the results for online LP algorithms are applicable to this problem. Without any assumption on $\gamma$, $C_{\mathcal{D}_2} \leq \frac{1}{e}$ since this problem contains the secretary problem. Korula and P{\'a}l \cite{korula2009algorithms} proposed an algorithm with $C_{\mathcal{D}_2} \geq \frac{1}{8}$. Their algorithm incorporates the optimal stopping rule for the classical secretary problem. Kesselhiem et al. \cite{kesselheim2013optimal} proposed an algorithm that incorporates the same stopping rule and achieves the optimal competitive ratio of $C_{\mathcal{D}_2} = \frac{1}{e}$.
%
%\item[-] Online Packing Problem:

\end{enumerate}

\section{Preliminaries}

In this section, we review basic concentration inequalities that will be used in the rest of this paper.

\begin{prop}[Bernstein's inequality]\label{BernsteinInequality}
Suppose $Y_i$, $1\leq i\leq n$ are independent random variables taking values in $\mathbb{R}$. Let $Z = \sum_{i=1}^{n}Y_i$, $\mu = E[Z]$, and $v =  \sum_{i=1}^{n}E[Y_i^2]$. Let $\gamma>0$ and suppose $Y_i \leq \gamma$ for all $i$, then for all $r\in [0,\, 3/\gamma)$:

\begin{align*}
&\log E\left[\exp{\left(r \left(Z - \mu \right)\right)}\right] \leq \frac{v r^2}{2(1-\gamma r/3)};
\end{align*}

\noindent therefore, for all $t>0$:

\begin{align*}
&\mathbb{P}\left(Z- \mu >  t \right) \leq \exp{\left(-\frac{ t^2}{2(v+\gamma t/3)}\right)}.
\end{align*}

\end{prop}

For proof see theorem 2.10 and corollary 2.11 of \cite{boucheron2013concentration}. Note that if $0 \leq  Y_i \leq \gamma$ for all $i$, then $v = \sum_{i=1}^{n}E[Y_i^2] \leq \gamma  \sum_{i=1}^{n}E[Y_i] = \gamma \mu $, which simplifies the Bernstein's inequality:

\begin{align}\label{Bernstein}
%&\log E\left[\exp{\left(r \left(Z - \mu \right)\right)}\right] \leq \gamma \mu r^2,\\
%&\log E\left[\exp{\left(r \left(-Z + \mu \right)\right)}\right] \leq \frac{\gamma \mu r^2}{2},\\
&\mathbb{P}\left(Z- \mu >  t \right) \leq \exp{\left(-\frac{t^2}{2\gamma(\mu + t/3)}\right)},\\
&\mathbb{P}\left(Z- \mu < - t \right) \leq \exp{\left(-\frac{ t^2}{2\gamma \mu}\right)},
\end{align}

\noindent for all $t > 0$. 

\begin{prop}[Hoeffding's theorem]\label{Hoeffding}
Let $X_1,\ldots,X_n$ be samples from a finite multiset $M \subset \mathbb{R}$ drawn without replacement, and let $Y_1,\ldots,Y_n$ be samples drawn with replacement from $M$. If $f$ is a convex function, then:

\begin{align}
E\left[f\left(\sum_{i=1}^{n}X_i\right)\right]\leq E\left[f\left(\sum_{i=1}^{n}Y_i\right)\right].
\end{align}
\end{prop}

Using proposition~\ref{Hoeffding}, one can apply the result of proposition~\ref{BernsteinInequality} to random variables sampled without replacement. Note that one can sharpen some of the bounds for the case of sampling without replacement; for example, see \cite{serfling1974probability}. 

\begin{prop}[Doob's maximal inequality]
Suppose $S_i$, $i = 1,2,\ldots, n$ is a sub-martingale. Let $M_n = \max_{1 \leq i \leq n}S_i$, then for all $t>0$:
\begin{align}
&t \mathbb{P}\left(M_n > t\right) \leq E \left\lvert S_n\right\lvert,\\
&\mathbb{P}\left(M_n > t\right) \leq \inf_{r>0}\frac{E\left[\exp{\left(r S_n\right)}\right]}{\exp{\left(r t\right)}}.
\end{align}
\end{prop}

For proof see chapter 8.10 \cite{shorack2000probability}. 

In the rest of this paper, $\Pi$ denotes the set of all the permutations of $\left[n\right]$, i.e., set of bijections from $[n]$ to $[n]$. For all $B \subset \Pi$, we denote the complement of $B$ with $\bar{B}$, and we define :

\begin{align}
&1_{B}(s) = \left\{\begin{array}{cc} 1 & s \in B,\\ 0 &  s \notin B.\end{array}\right.
\end{align}

We denote the transpose of a vector $v$ by $v^{\top}$. For a function $f: \mathbb{R}^k \mapsto [-\infty, +\infty)$, $f^*$ denotes the concave conjugate of $f$, and is defined as:

\begin{align*}
f^*(v) = \inf_{x} v^\top x - f(x),
\end{align*}

\noindent for all $v \in \mathbb{R}^k$. For a concave function $f$, $\partial f(x)$ denotes the set of supergradients of $f$ at $x$, i.e., the set of all $v \in \mathbb{R}^k$ such that:

\begin{align*}
\forall x' \in \mathbb{R}^k: \quad f(x') \leq v^{\top} (x' - x) + f(x).
\end{align*}

For a convex function $f$, we use the same notation $\partial f(x)$ to denote the set of  subgradients of $f$ at $x$, i.e., the set of all $v \in \mathbb{R}^k$ such that:

\begin{align*}
\forall x' \in \mathbb{R}^k: \quad f(x') \geq v^{\top} (x' - x) + f(x).
\end{align*}
%Finally, We state three elementary inequalities that will be used in this paper. 
%
%\begin{align}\label{el1}
%&\forall u > 0, \forall \alpha \in (\frac{1}{2}, 1] : \quad  u - (1+\alpha  u)\log(1+ u) \leq \frac{- b u^2}{1+c u},\\ \notag
%&\text{where   } b = \alpha - \frac{1}{2}, \;c = \frac{5\alpha - 3}{9 \alpha - 3},\\ \label{el2}
%%&\forall u \in [0,1], \forall \alpha\geq 1 :  \quad- u - (1-\alpha  u)\log(1- u) \leq \left(\frac{1}{2}-\alpha\right) u^2,\\ \label{el3}
%&\forall u \in [0,1]:  \quad u + (1-  u)\log(1- u) \geq \frac{1}{2} u^2,\\ \label{el3}
%&\forall \nu  \in \mathbb{R}, \forall x \in [0,1]:   \quad e^{\nu x} \leq 1+ x ( e^{\nu} - 1).
%\end{align}

\section{A feasibility problem}\label{sec::simple}
We start by considering a simple form of ESA applied to a linear feasibility problem to illustrate the proof technique and ideas. Let $\Delta = \{x \in \mathbb{R}^k | \mathbf{1}^{\top} x = 1, \; x \geq 0\}$. Consider a spacial case of \eqref{vector_general_i}, where $b = \mathbf{1} \in \mathbb{R}^m$, and for all $t\in [n]$, $f_t(x) = I_{x \in \Delta}$. This problem can be represented as the following feasibility problem:

% and we aim to choose $x_t$ in order to minimize 
%$\max_i \sum_{t=1}^{n} \left(A_t x_t\right)_i$:
%
%In this section, we analyze a simpler form of ESA applied to a simple LP. The main ideas used in the analysis of ESA in the next section are similar to those presented in this section.  Let $\Delta = \{x \in \mathbb{R}^n | \mathbf{1}^{\top} x = 1, \; x \geq 0\}$. Consider the following optimization problem:

\begin{align}\label{feasibility}
\begin{array}{llll}
\mbox{Find}&{   (x_1,x_2,\ldots,x_n) \in \Delta^n }&\mbox{such that} &\sum_{t=1}^{n}{A_t x_t } \leq  \mathbf{1},
\end{array}
\end{align}

\noindent where for all $t$, $A_t \in \mathbb{R}_{+}^{m\times k}$. Let $\left(x^*_1,x^*_2,\ldots,x^*_n\right)$ be a feasible solution for \eqref{feasibility}. Define $\gamma = \max_{t,i,j} A_{t,i,j}$ and choose $\sigma$ uniformly at random from $\Pi$. Consider Algorithm 1 for problem \eqref{feasibility}.
% We analyze the competitive ratio of the following algorithm under the random permutation 
%This algorithm is similar to the online algorithm proposed by Devanur et al. \cite{devanur2011near} for stochastic setting .
\begin{algorithm}
\caption{}\label{simple}
\begin{algorithmic}
\REQUIRE{$\epsilon \in (0,1)$, $\gamma > 0$}
\STATE{$\nu = \frac{1}{\gamma}\log{\left(1+\e\right)}$}  
\FOR{$t = 1:n$}   
%\STATE{Inquire $A_{\s(t)}$} 
\STATE{$x_t \in \argmin_{z \in \Delta}{\sum_{i=1}^{m}{ \exp{\left({\nu} \sum_{s=1}^{t-1}\left(A_{\sigma\left(s\right)}x_s\right)_i \right)}\left(A_{\sigma\left(t\right)}z\right)_i}}$}
\ENDFOR
\end{algorithmic}
\end{algorithm}

Algorithm \ref{simple} is proposed by Devanur et al. \cite{devanur2011near}. Here we extend the analysis of algorithm \ref{simple} to the random permutation model. We first derive concentration inequalities for the running sums of form ${\sum_{s=t}^{n}\left(A_{\sigma\left(s\right)} x^*_{\sigma\left(s\right)}\right)_i }$. These inequalities will be used in the analysis of the algorithm, where we compare $\left(x_1,x_2,\ldots,x_n\right)$, the solution given by the algorithm, with $\left(x^*_1,x^*_2,\ldots,x^*_n\right)$. Let $L = \log_2(\frac{1}{\e})$. To simplify the notation, we assume that $L$ and $n \e$ are integers in the rest of this paper. Let $\k: [n(1-\e)]\mapsto [L]$ be defined as $\k(t) = \left\lfloor \log_2\left(\frac{n-t}{n\epsilon}\right)\right\rfloor + 1 $. Equivalently, $\k(t) = k$  if and only if $n (1-2^{k}\epsilon ) < t \leq n \left(1-2^{k-1}\epsilon\right)$ with $k \in [L]$. Furthermore, define:

\begin{align*}
%t\in [T]&: \k(t) = \left\lfloor \log\left(\frac{n-t}{n\epsilon}\right)\right\rfloor + 1\\
\forall t \in \left[n\right], \forall i\in \left[m\right]&: {X_t^*}^i = \left(A_{\sigma\left(t\right)} x^*_{\sigma\left(t\right)}\right)_i,\\
\forall t \in \left[n\right], \forall i\in \left[m\right]&: X_t^i = \left(A_{\sigma\left(t\right)} x_{t}\right)_i,\\
\forall t \in [n]&: R_t^i = \frac{\sum_{s=t}^{n}{X_s^*}^i }{n-t+1} - \frac{1}{n}, \\
\forall t \in [n]&:  B_t= \bigcap_{i=1}^{m}\left\{\sigma\left\lvert R_t^i\leq \frac{ 1}{n} 2^{-\frac{\k(t)}{2}}\epsilon^{\frac{1}{2}}\right.\right\},\\
\forall k \in [L]&: M_k^i = \max_{ n(1-2^{k}\epsilon)< t \leq n(1-2^{k-1}\epsilon)}{R_t^i}. 
\end{align*}

Note that $E\left[\frac{1 }{n-t+1}\sum_{s=t}^{n}{X_s^*}^i\right] \leq \frac{1}{n}$ for all $i \in [m]$ and $t \in \left[n\right]$. The event $B_t$ consists of all the permutations for which $\max_i \frac{1 }{n-t+1}\sum_{s=t}^{n}{X_s^*}^i$ does not exceed $\frac{1}{n}$ by more than a small fraction of $\frac{1}{n}$. The purpose of the next two paragraphs is to bound the probability of $\cup_{1}^{n(1-\e)}\bar{B}_t$. To do so, we first show that for all $i$, $R_t^i$ is a martingale:

\begin{align*}
E[R_t^i| R_{t-1}^i, \ldots, R_{1}^i] &= E\left[\left.\frac{\sum_{s=t-1}^{n}{{X_s^*}^i } - {X^*}_{t-1}^i}{n-t+1} - \frac{1}{n}\right\lvert R_{t-1}^i, \ldots, R_{1}^i \right]\\
&= \frac{\sum_{s=t-1}^{n}{{X_s^*}^i }}{n-t+1}  - \frac{1}{n}- \frac{1}{n-t+1}   E\left[\left.{X^*}_{t-1}^i\right\lvert R_{t-1}^i, \ldots, R_{1}^i \right] \\
&= \frac{\sum_{s=t-1}^{n}{{X_s^*}^i }}{n-t+1}  - \frac{1}{n}- \frac{\sum_{s=t-1}^{n}{{X_s^*}^i }}{(n-t+1)(n-t+2)} =  R_{t-1}^i.
\end{align*}

By Doob's maximal inequality and Bernstein's inequality:

\begin{align}
\forall k \in [L]:&\quad \mathbb{P}\left(M_k^i > \frac{2^{-\frac{k}{2}}\epsilon^{\frac{1}{2}}}{n}\right) \leq \inf_{r \geq 0}{E\left[\exp{\left(r R_{n(1-2^{k-1}\epsilon)}^i \right)}\right]  \exp{\left(-\frac{r 2^{-\frac{k}{2}}\epsilon}{n}\right)}}\leq  \exp{\left(\frac{-\epsilon^2}{6 \gamma }\right)},\\ \label{prob-bound}
%\forall k \in [L]:&\quad  \mathbb{P}\left(\bigcup\limits_{n(1-2^{k}\epsilon)}^{ n(1-2^{k-1}\epsilon)}\bar{B}_t\right) = \mathbb{P}\left(\max_i M_k^i > \frac{2^{-\frac{k}{2}}  \epsilon^{\frac{1}{2}}}{n}\right) \leq m \exp{\left(\frac{-\epsilon^2}{6 \gamma }\right)}.\\
\Rightarrow \quad &\mathbb{P}\left(\bigcup\limits_{t=1}^{n(1-\epsilon)} \bar{B}_t\right) \leq \sum_{k=1}^{L} \mathbb{P}\left(\bigcup\limits_{n(1-2^{k}\epsilon)}^{ n(1-2^{k-1}\epsilon)}\bar{B}_t\right) = \sum_{k=1}^{L}\mathbb{P}\left(\max_i M_k^i > \frac{2^{-\frac{k}{2}}  \epsilon^{\frac{1}{2}}}{n}\right) \leq m L \exp{\left(\frac{-\epsilon^2}{6 \gamma }\right)}.
\end{align}

\begin{thm}\label{thm1}
Suppose $x_t$, $1\leq t \leq n$ is given by algorithm \ref{simple}. Let $T= n(1-\epsilon)$.  If  $\e \leq \frac{1}{4}$ and $\gamma \leq \frac{\epsilon^2}{12 \log\left(m/\epsilon\right)}$, then $\mathbb{P}\left(\max_i \sum_{s=1}^{T}X_s^i > (1+2\e) \right) \leq  \e$.
\end{thm}

\begin{proof}

Define:

\begin{align*}
\forall t\in [T]&: \beta_t=\frac{\e}{\gamma n}\left({1 + 2^{-\frac{\k(t)}{2}}\epsilon^{\frac{1}{2}}}\right),\\
\forall t \in [T], \forall i \in [m] &: \phi^t_i = {{ \exp{\left(\sum_{s=1}^{t}{\nu}\left(A_{\sigma\left(s\right)}x_s\right)_i  - \beta_s \right)}}},\\
\forall t \in [T]&:{\Phi^{t}} = \sum_{i=1}^{m} {\phi^{t}_i}\prod_{s=1}^{t}1_{B_s}.
\end{align*}

For all $i \in [m]$, define $\phi_i^0 = 1$ and $\Phi^0 = \sum_{i=1}^{m} \phi_i^0 = m$. Let $\mathcal{A}_0 = \{\varnothing, \Pi \}$, and let $\mathcal{A}_t$ be the sigma algebra generated by $\sigma(1),\ldots, \sigma(t)$.  We first show that $\left\{ \Phi_t, {\cal A}_t\right\}$, $t \in \{0,1,\ldots, T\}$ is a super-martingale:

\begin{align*}
 E\left[\left.{\Phi^{t}}\right\lvert \mathcal{A}_{t-1} \right]&= E\left [\left.\left(\sum_{i=1}^{m} \phi_i^{t-1} \exp{\left(\nu X_{t}^i- \beta_t \right)}\right)\prod_{s=1}^{t}1_{B_s}\right\lvert \mathcal{A}_{t-1} \right]\\
&\leq E\left [\left.\left(\sum_{i=1}^{m} \phi_i^{t-1} \left(1+\frac{\e }{\gamma}X_{t}^i\right)\exp{\left(- \beta_t \right)}\right)\prod_{s=1}^{t}1_{B_s}\right\lvert \mathcal{A}_{t-1} \right]\\
&\leq E\left [\left.\left(\sum_{i=1}^{m} \phi_i^{t-1} \left(1+\frac{\e}{\gamma} {X_t^*}^i\right)\exp{\left(- \beta_t \right)}\right)\prod_{s=1}^{t}1_{B_s}\right\lvert \mathcal{A}_{t-1} \right]\\
&= \sum_{i=1}^{m} \phi_i^{t-1} E\left[\left. 1+\frac{\e}{\gamma} {X_t^*}^i \right\lvert \mathcal{A}_{t-1}\right] \exp{\left(- \beta_t \right)} \prod_{s=1}^{t}1_{B_s}\\
&\leq \sum_{i=1}^{m} \phi_i^{t-1}  1_{\left\{R_t^i\leq \frac{1}{n}{2^{-{\k(t)}/{2}}\epsilon^{{1}/{2}}}\right\}}  \exp{\left(\frac{\e}{\gamma}\left(\frac{1}{n} + R_t^i \right)-\beta_t \right)}\prod_{s=1}^{t-1}1_{B_s}\\
&\leq \sum_{i=1}^{m} \phi_i^{t-1}  \prod_{s=1}^{t-1}1_{B_s} = \Phi^{t-1}.
%&& \text{by \eqref{Bernstein}}
\end{align*}

The first inequality follows from:

\begin{align} \label{el3}
&\forall \nu  \in \mathbb{R}, \forall x \in [0,1]:   \quad e^{\nu x} \leq 1+ x ( e^{\nu} - 1).
\end{align}

The second inequality follows from the fact that

 $$x_t \in \argmin_{\begin{array}{c}z \in \Delta\end{array}}{\sum_{i=1}^{m}{ \exp{\left({\nu} \sum_{s=1}^{t-1}\left(A_{\sigma\left(s\right)}x_s\right)_i \right)}\left(A_{\sigma\left(t\right)}z\right)_i}},$$

\noindent and thus $x_t$ as the minimizer achieves a smaller objective value than any other $x\in \Delta$ including $x^*_{\s(t)}$. Let $F =  \cap_{s=1}^{T} {B}_s $. The previous result combined with Markov's inequality yields:

\begin{align}\notag
%-----------------------------------------------------------------------------------------------------------------------------------------------------------------------
\mathbb{P}\left(\left\{\max_i \sum_{s=1}^{T}X_s^i > \frac{T}{n}\left(1+3\epsilon\right)\right\}\cap F\right) &= \mathbb{P}\left(\left\{\max_i \phi_i^T > \exp{\left(\nu \frac{T}{n}\left(1+3\epsilon\right)-\sum_{s=1}^{T}   \beta_s \right)}\right\}\cap F\right)\\\notag
%-----------------------------------------------------------------------------------------------------------------------------------------------------------------------
&\leq \mathbb{P}\left(\left\{\sum_{i=1}^{m} \phi_i^T >  \exp{\left(\nu \frac{T}{n}\left(1+3\epsilon\right)-\sum_{s=1}^{T}   \beta_s \right)}\right\}\cap F\right)\\\notag
%-----------------------------------------------------------------------------------------------------------------------------------------------------------------------
&\leq \mathbb{P}\left(\left\{\sum_{i=1}^{m} \phi_i^T \prod_{t=1}^{T}1_{B_t}> \exp{\left(\nu \frac{T}{n}\left(1+3\epsilon\right) -\sum_{s=1}^{T}   \beta_s \right)}\right\}\right) \\\notag
%-----------------------------------------------------------------------------------------------------------------------------------------------------------------------
%&\leq \mathbb{P}\left(\left\{\Phi^T> \exp{\left(\frac{T}{n}\nu\left(1+3\epsilon\right) -\sum_{s=1}^{T}   \beta_s \right)}\right\}\right) \\
%-----------------------------------------------------------------------------------------------------------------------------------------------------------------------
%&\leq E \left[\Phi^{T}\right] \exp{\left(\sum_{s=1}^{T}  \beta_s  -  \frac{T}{n}\nu \left(1+3\epsilon\right)\right)}\\
%-----------------------------------------------------------------------------------------------------------------------------------------------------------------------
&\leq E\left[\Phi^{T}\right]\exp{\left(\frac{T}{n \gamma}\left(\e \left(1+2\epsilon\right) - \left(1+3\epsilon\right) {\nu}{\gamma}\right)\right)} \\\notag
%-----------------------------------------------------------------------------------------------------------------------------------------------------------------------
&\leq E\left[\Phi^0\right]\exp{\left(\frac{T\left(1+2\epsilon\right)}{n \gamma}\left(\epsilon-\left(1+\frac{ \epsilon}{1+2\e}\right)\log{\left(1+\e\right)}\right)\right)}  \\ \label{fuck}
%-----------------------------------------------------------------------------------------------------------------------------------------------------------------------
&\leq  m\exp{\left(\frac{-(1-\e)(1+2\e)\epsilon^2}{6\gamma(1+\e/3)}\right)}\leq  m\exp{\left(\frac{-\epsilon^2}{6\gamma}\right)}.
\end{align}

In the last line, we used the following inequality:

\begin{align}\label{el1}
&\forall u > 0, \forall \alpha \in \left[\frac{1}{2}, 1\right] : \quad  u - (1+\alpha  u)\log(1+ u) \leq \frac{- (\alpha - 1/2) u^2}{1+ u/3}.
\end{align}

Using the bounds given by \eqref{prob-bound} and \eqref{fuck}, we can conclude that:

\begin{align*}
%-----------------------------------------------------------------------------------------------------------------------------------------------------------------------
\mathbb{P}\left(\max_i \sum_{s=1}^{T}X_s^i > \left(1+2\epsilon\right)\right) &\leq \mathbb{P}\left(\max_i \sum_{s=1}^{T}X_s^i > \frac{T}{n}\left(1+3\epsilon\right)\right) \\
%-----------------------------------------------------------------------------------------------------------------------------------------------------------------------
 &\leq \mathbb{P}\left(\left\{\max_i \sum_{s=1}^{T}X_s^i > \frac{T}{n}\left(1+3\epsilon\right)\right\}\cap F\right) + \mathbb{P}\left(\bar{F}\right)\\
%-----------------------------------------------------------------------------------------------------------------------------------------------------------------------
% &\leq m\exp{\left(\frac{-\epsilon^2}{6\gamma }\right)}  + \mathbb{P}\left(\bigcup\limits_{s=1}^{n(1-\epsilon)} {\bar{B}}_s \right)\\
%-----------------------------------------------------------------------------------------------------------------------------------------------------------------------
% &\leq m\exp{\left(\frac{-\epsilon^2}{6\gamma }\right)} + \sum_{k=1}^{L} \mathbb{P}\left(\bigcup\limits_{n(1-2^{k}\epsilon)}^{ n(1-2^{k-1}\epsilon)}\bar{B}_t\right)\\
%-----------------------------------------------------------------------------------------------------------------------------------------------------------------------
&\leq m(L+1) \exp{\left(\frac{-\epsilon^2}{6 \gamma }\right)} \leq  \e ,%  log(x) \leq x^{0.375}
\end{align*}

\noindent where the second inequality follows from the fact that for any two events $E$ and $F$, $\mathbb{P}(E) = \mathbb{P}(E \cap F) + \mathbb{P}(E \cap \bar{F}) \leq \mathbb{P}(E \cap F) + \mathbb{P}( \bar{F})$. In the last line, we used the fact that $L = \log_2(\frac{1}{\e})\leq \log_2(e) \e^{-\frac{1}{e}}$ alongside the assumption that $\e \leq \frac{1}{4}$.

\end{proof}

In the next section, we introduce ESA in its full generality and provide a competitive ratio analysis for ESA under the random permutation model.

\section{General problem}\label{general}

Recall that the offline optimization problem is as follows:

\begin{align}\tag{$\mathcal{P}1$}\label{vector_general}
P^* = \sup_{x_1,\ldots,x_n\in \mathbb{R}^k}&{  \quad \sum_{t=1}^{n} {f_t(x_t)}+\psi\left(\sum_{t=1}^{n}{A_t x_t }\right)},
\end{align}

\noindent where for all $t \in [n]$, $A_t \in \mathbb{R}_+^{m\times k}$, $f_t \in \mathcal{G}$, and $\psi(u) =  I_{u \leq b}$ for some $b \in  \mathbb{R}_{++}^m$. Here $\mathcal{G}$ is a set of proper, concave, and upper semi-continuous functions on $\mathbb{R}^k$ with bounded super-level sets such that for all $f \in \mathcal{G}$, ${\rm dom} f \subset \mathbb{R}_{+}^k$ and $f(0) = 0$. 

The Fenchel dual program is:

\begin{align}\tag{$\mathcal{D}1$}
D^* = \inf_{y \in \mathbb{R}^m}&{  \quad \sum_{t=1}^{n} {-f^*_t(A_t^\top y)} - \psi^*\left(-y \right)}.
\end{align}

A pair of primal and dual variables $\left(\left(x^*_1,x^*_2,\ldots,x^*_n\right),y^*\right)$ is an optimal pair when:

\begin{align} 
&  y^* \in  -\partial \psi\left(\sum_{t=1}^{n}{A_t x_t^* } \right)\label{cond_g_1},\\ \notag
\\
&\forall t \in [n]: \quad A_t^\top y^* \in  \partial{f_t}\left(x_t^*\right). \label{cond_g_2}  
\end{align}

By our assumption on the functions involved such a pair exists and $0 \leq P^* = D^* < +\infty$. We assume that $P^* > 0$. Let $\gamma$ be such that:
% for all $t \in [n]$, $\left(f_t,A_t\right)$ satisfy:

\begin{align}\label{gamma}
\forall t\in[n]: \quad \gamma \geq \max\left(\left\{\left.\frac{\left(A_t x\right)_i}{b_i}\right\lvert i \in [m], f_t(x) \geq 0\right\} \cup \left\{\left.\frac{f_t(x)}{P^*}\right\lvert x \in {\rm dom}f_t\right\}\right).
\end{align}
%$\gamma =\max\{ \max_{i,t}{\frac{\left(A_t x_t^*\right)_i}{b_i}} , \max_t \frac{f_t(x^*_t)}{P^*}\}$. 

Note that for all $t \in [n]$:

\begin{align*}
 &0\leq f_t(x^*_t)  ,\\
&0 \leq - f^*_t(A_t^\top y^*) = f_t(x^*_t) - {y^*}^{\top} A_t x^*_t \leq \gamma P^*.
\end{align*}

ESA is designed to solve the following problem which is equivalent to \eqref{vector_general}:

%To accurately represent the dual variable updates in ESA as subgradient steps, we consider the following equivalent problem to \eqref{vector_general}:

%\begin{align}\tag{$\mathcal{P}2$}\label{vector_general_equiv}
%P^* = \sup_{\begin{array}{c}x_t \in \mathbb{R}^{k} \end{array}}&{P^*+\hat{\psi}\left(-\sum_{t=1}^{n} {f_t(x_t)}\right)+\psi\left(\sum_{t=1}^{n}{A_t x_t }\right)},\\ \notag
%\end{align}

\begin{align}\tag{$\mathcal{P}2$}\label{vector_general_equiv}
\underset{x_1,\ldots,x_n\in \mathbb{R}^k}{\mbox{maximize}}&{\quad \hat{\psi}\left(-\sum_{t=1}^{n} {f_t(x_t)}\right)+\psi\left(\sum_{t=1}^{n}{A_t x_t }\right)},\end{align}

\noindent where $\hat{\psi}(v) = I_{v \leq - P^*}$. Since $P^*$ cannot be calculated without having access to $\left(\left(f_1,A_1\right),\left(f_2,A_2\right),\ldots,\left(f_n,A_n\right)\right)$, ESA retains an estimate of $P^*$ and updates that estimate in exponential intervals. Choose $\sigma$ uniformly at random from $\Pi$. Fix $\epsilon \in (0  ,  1)$ and let $h \in \{0,\ldots, L-1\}$. To estimate the optimal value using $(f_{\s(t)}, A_{\s(t)})$, $1\leq t\leq 2^h n\epsilon$, consider the following optimization problem:

\begin{align}\label{Ph}
P_h &= \sup_{x_t \in \mathbb{R}^{k}\;\; t \in \left[2^h n \e\right ]}{  \quad\frac{1}{2^{h}\epsilon(1-\theta_{h})} \sum_{t=1}^{2^h n\epsilon} {f_{\sigma(t)}(x_t)}+\psi\left(\frac{1}{2^{h} \epsilon(1+\theta_{h})} \sum_{t=1}^{2^h n\epsilon}{A_{\sigma(t)} x_t }\right)}\\ \notag
&\leq \inf_{y \in \mathbb{R}^m}{  \quad \frac{1}{2^{h} \epsilon(1-\theta_{h})}\sum_{t=1}^{2^h n\epsilon} {-f^*_{\sigma(t)}\left( A_{\sigma(t)}^{\top} y\right)} - \psi^*\left(-\frac{1+\theta_h}{1-\theta_h}y \right)},
\end{align}

\noindent where $\theta_h = 2^{-\frac{h+1}{2}}\epsilon^{\frac{1}{2}}$. In order to compare $P_h$ with $P^*$, we define:
\begin{align}\notag
\tilde{P}_h &= {  \frac{1}{2^{h}\epsilon(1-\theta_{h})} \sum_{t=1}^{2^h n\epsilon} {f_{\sigma(t)}(x^*_{\sigma(t)})}+\psi\left(\frac{1}{2^h\epsilon(1+\theta_{h})} \sum_{t=1}^{2^h n\epsilon}{A_{\sigma(t)} x^*_{\sigma(t)} }\right)},\\ \notag
\tilde{D}_h&= {   \frac{1}{2^h\epsilon(1-\theta_{h})}\sum_{t=1}^{2^h n\epsilon} {-f^*_{\sigma(t)}\left(A_{\sigma(t)}^{\top} y^*\right)} - \psi^*\left(-\frac{1+\theta_h}{1-\theta_h} y^* \right)}.
\end{align}

Note that $\tilde{D}_h \geq P_h \geq \tilde{P}_h$. Now using the previous fact alongside Bernstein's inequality, we show that $P_h$ is close to $P^*$ with high probability. 

\begin{align}\notag
\mathbb{P}\left({P_h}< P^*\right) &\leq \mathbb{P}\left(\tilde{P}_h <  P^*\right) \\\notag
&\leq \mathbb{P}\left( \sum_{t=1}^{2^h n\epsilon} {f_{\sigma(t)}(x^*_{\sigma(t)})} < 2^h\epsilon (1-\theta_{h}) P^* \right) + \sum_{i=1}^{m}{\mathbb{P}\left(\sum_{t=1}^{2^h n\epsilon}{(A_{\sigma(t)} x^*_{\sigma(t)})_i } > 2^h \epsilon(1+\theta_{h}) b_i \right)}\\\label{Ph-bound}
&\leq \exp{\left(\frac{-\epsilon^2}{4 \gamma}\right)} + m \exp{\left(\frac{-\epsilon^2}{6 \gamma}\right)},
\end{align}

\begin{align}\notag
\mathbb{P}\left(P_h >\frac{1+2 \theta_{h}}{1-\theta_{h}}  P^*\right) &\leq \mathbb{P}\left(\tilde{D}_h> \frac{1+ 2 \theta_{h}}{1-\theta_{h}}   P^*\right) \\\label{Dh-bound}
&\leq \mathbb{P}\left(  \sum_{t=1}^{2^h n\epsilon} {- f^*_{\sigma(t)}(A_{\sigma(t)}^{\top} y^*)} > -2^h\epsilon (1+\theta_{h}) \sum_{t=1}^{n} {f^*_{\sigma(t)}(A_{\sigma(t)}^{\top} y^*)} + 2^h \e \theta_h P^* \right) \leq \exp{\left(\frac{-\epsilon^2}{6\gamma}\right)}.
\end{align}

Let $\k$ be defined as in section~\ref{sec::simple} and let $\eta: \{n\e+1, n\e+2, \ldots, n\} \mapsto [L]$ be defined as $\eta(t) = \k(n-t+1)$. Equivalently, $\eta(t) = h$  if and only if $n 2^{h-1}\epsilon < t \leq  n 2^{h}\epsilon$ with $h \in [L]$. For all $t \in \{n\e+1, n\e+2, \ldots, n\}$, let $q_t = P_{\eta{(t)}-1}$ if $P_{\eta{(t)}-1} \neq 0$. To avoid division by zero, we set $q_t = \e$ if $P_{\eta{(t)}-1} = 0$. Note that since $P^* >0$, the probability of $P_{\eta{(t)}-1} = 0$ is bounded by \eqref{Ph-bound}. Now we define parameters that will be used in ESA. For all $t\in \{n\e+1, n\e+2,  \ldots, n(1-\e)\}$, define:

\begin{align*}
&\alpha_t =  \frac{1-2^{-\frac{\eta(t)}{2}}\epsilon^{\frac{1}{2}}}{1+2^{-\frac{\eta(t)}{2}+1}\epsilon^{\frac{1}{2}}}, \; \beta_t=\frac{\e}{\gamma n}\left(1+2^{-\frac{\k(t)}{2}}\epsilon^{\frac{1}{2}}\right),\;\beta'_t = \frac{\e}{\gamma n}\left(1-2^{-\frac{\k(t)}{2}}\epsilon^{\frac{1}{2}}\right)\alpha_t,\\
\end{align*}

\noindent while for all $t \in \{n(1-\e)+1, \ldots, n+1\}$, let $\beta_t = \beta_{n(1-\e)}$ and $\beta'_t = \beta'_{n(1-\e)}$.

\begin{algorithm}
\caption{Exponentiated Subgradient Algorithm (ESA)}\label{general-algorithm}
\begin{algorithmic}
\REQUIRE{$\epsilon \in (0,1)$, $\gamma > 0$}

\STATE{$\nu = \frac{1}{\gamma}\log{\left({1+\epsilon}\right)}$, $\nu\rq{} = \frac{-1}{\gamma}\log{\left({1-\epsilon}\right)}$}

\FOR{$t = 1:n\e$}   

%\STATE{${x}_t = 0$}
\STATE{$\check{x}_t = 0$}

\ENDFOR

\STATE{$y'_{n\e+1} = \exp{(\beta'_{n\e+1})}$, $y_{n\e+1,i}= \exp{(-\beta_{n\e+1})}/m \quad \left(\forall i\in [m]\right)$}
%\STATE{$x_{n\e} = 0$}

\FOR{$t = n \epsilon + 1 :  n$}   

\STATE{$x_t \in \argmin_{z}{\sum\limits_{i=1}^{m}{\frac{y_{t,i}}{ b_i}  \left(A_{\sigma\left(t\right)}z\right)_i } -\frac{y'_t}{q_t}  f_{\sigma(t)}\left(z\right) }$}
\STATE{$   y_{t+1,i} = y_{t,i} \exp{\left( \frac{\nu}{b_i}\left(A_{\sigma\left(t\right)}x_{t}\right)_i  -   \beta_{t+1}\right)} \quad \left(\forall i \in [m]\right) $}
\STATE{$        y'_{t+1} = y'_{t} \exp{\left(\frac{-\nu'}{q_{t}} f_{\sigma(t)}\left(x_{t}\right) + \beta'_{t+1}\right)} $}

%\IF{$\psi\left(\sum\limits_{s=n\e}^{t}A_{\sigma\left(s\right)}x_s\right) = 0$} 
\IF{$\sum\limits_{s=n\e+1}^{t}A_{\sigma\left(s\right)}x_s\leq b$} 
\STATE{$\check{x}_t = x_t$}
\ELSE
\STATE{$\check{x}_t = \mathbf{0}$}
\ENDIF

\ENDFOR
\end{algorithmic}
\end{algorithm}
%For all $t \in [n]$ let $y_t= \left[\frac{q_i}{b_i}  \exp{\left(\frac{\nu}{b_i} \sum_{s=1}^{t-1}\left(A_{\sigma\left(s\right)}x_s\right)_i  - \e \frac{\beta_s}{\gamma }\right)}\right]_{i=1}^{m}$, $y'_t = \frac{m}{q_s} \exp{\left(- \frac{\nu'}{q_s} \sum_{s=1}^{t-1} f_{\sigma(s)}\left(x_s\right) +\e \frac{\beta'_s}{\gamma }\right)}$, and $y''_t = \frac{1}{y'_t} y_t$. 
Algorithm \eqref{general-algorithm} describes ESA applied to \eqref{vector_general}. In section \ref{sec::connection}, we represent each update of $\left(y_t,y'_t\right)$ as an exponentiated subgradient step. Let $x_t, 1\leq t\leq n$ be given by algorithm \eqref{general-algorithm}.  For all $t \in [n]$, define:

%\begin{align*}
%y_t&= \left[\frac{q_s}{m b_i}  \exp{\left(\frac{\nu}{b_i} \sum_{s=n\e}^{t-1}\left(A_{\sigma\left(s\right)}x_s\right)_i  -  \sum_{s=n\e+1}^{t} \beta_s\right)}\right]_{i=1}^{m},\\
%y'_t &=  \exp{\left(- \frac{\nu'}{q_s} \sum_{s=n\e}^{t-1} f_{\sigma(s)}\left(x_s\right) + \sum_{s=n\e+1}^{t} \beta'_s\right)},
%\end{align*}

\begin{align*}
y''_t = \frac{q_t}{ y'_t} \left[{y_{t,1}}/{ b_1 },\,{y_{t,2}}/{ b_2 },\, \ldots, {y_{t,m}}/{ b_m }\right]^{\top}.
\end{align*}

%\noindent and $y''_t = \frac{1}{y'_t} y_t$. 

Note that $x_t \in \argmin {y''_t}^{\top} A_{\sigma\left(t\right)}z  - f_{\sigma(t)}\left(z\right) $ if and only if $  A_{\sigma\left(t\right)}^{\top} y''_t \in \partial f_{\sigma(t)}\left(x_t\right) $ if and only if $x_t \in \partial f^*_{\sigma(t)}\left(  A_{\sigma\left(t\right)}^{\top} y''_t\right) $. For special cases of $f_{\s(t)}$, this assignment rule can have a simple form; for example, if $k=1$ and $f_{\s(t)}(x) = \tilde{f}(x) + I_{x\in [0,1]}$ with $\tilde{f}$ differentiable, then $x_t$ can be chosen as:

\begin{equation}\label{assignment} \notag
{x}_t =   \left\{ \begin{array}{l c}  0 &    \tilde{f}'(0) \leq A_{\s(t)}^{\top} y''_t ,\\ 1 &  \tilde{f}'(1) >  A_{\s(t)}^{\top} y''_t, \\ l(A_{\s(t)}^{\top} y''_t) & \text{otherwise},\end{array} \right.
\end{equation}

\noindent where $l(v) = \min{\{x |v\geq  \tilde{f}'(x)\}}$. Similar to the analysis in section~\ref{sec::simple}, we need to derive concentration inequalities for the running sums of form $\sum_{s=t}^{n}\left(A_{\sigma(s)}x^*_{\sigma(s)}\right)_i$ and $\sum_{s=t}^{n}f_{\sigma(s)}\left(x^*_{\sigma(s)}\right)$. Define:

\begin{align*}
\forall i\in [m], \, \forall t \in [n]&: \quad  {X_t^*}^i = \left(A_{\sigma(t)}x^*_{\sigma(t)}\right)_i,\\
\forall i\in [m],\,\forall t\in [n]&: \quad R_t^i = \frac{\sum_{s=t}^{n}{X_s^*}^i }{n-t+1} - \frac{b_i}{n} ,\\
\forall t\in [n]&: \quad S_t = \frac{\sum_{s=t}^{n}f_{\sigma(s)}\left(x^*_{\sigma(s)}\right)}{n-t+1} - \frac{P^*}{n}, \\
\forall t \in [n]&: \quad C_t = \left\{\sigma\left\lvert S_t\geq  -\frac{1}{n}{P^*2^{-\frac{\k(t)}{2}}\epsilon^{\frac{1}{2}}}\right.\right\},\\
\forall t \in [n]&: \quad B_t = \bigcap_{i=1}^{m}\left\{\s\left\lvert R_t^i\leq \frac{1}{n}b_i 2^{-\frac{\k(t)}{2}}\epsilon^{\frac{1}{2}}\right.\right\},\\
\forall i \in [m], \; \forall k \in [L]&: \quad M^i_k = \max_{n (1-2^{k}\epsilon ) < t \leq n \left(1-2^{k-1}\epsilon\right)}R_t^i,\\
\forall k \in [L]&: \quad N_k = \max_{n (1-2^{k}\epsilon ) < t \leq n \left(1-2^{k-1}\epsilon\right)} - S_t.\\
\end{align*}

As in section \ref{sec::simple}, for all $i$, $R_t^i$ is a martingale. Similarly, it can be shown that $S_t$ is a martingale. By Doob's maximal inequality and Bernstein's inequality:

\begin{align}\notag
\forall i\in [m],\;\forall k \in  [L]: \quad &\mathbb{P}\left(M^i_k > \frac{1}{n}{b_i 2^{-\frac{k}{2}}\epsilon^{\frac{1}{2}}}\right) \leq \inf_{r \geq 0}{E\left[\exp{\left(r R_{n(1-2^{k-1}\epsilon)}^i\right)}\right]}{\exp{\left(-\frac{1}{n}{r b_i 2^{-\frac{k}{2}}\epsilon^{\frac{1}{2}}}\right)}} \leq  \exp{\left(\frac{-\epsilon^2}{6\gamma }\right)}\\\label{m-bound}
%\Rightarrow \quad &\mathbb{P}\left(\bigcup\limits_{t=1}^{n(1-\epsilon)} \bar{B}_t\right)\leq  = \mathbb{P}\left(\max_{i}\max_{k}\frac{2^{\frac{k}{2}} M^i_k}{\epsilon^{\frac{1}{2}} b_i} > \frac{1}{n}\right) \leq m L\exp{\left(\frac{-\epsilon^2}{6 \gamma }\right)}.\\\label{m-bound}
\Rightarrow \quad &\mathbb{P}\left(\bigcup\limits_{t=1}^{n(1-\epsilon)} \bar{B}_t\right) \leq \sum_{k=1}^{L} \mathbb{P}\left(\bigcup\limits_{n(1-2^{k}\epsilon)}^{ n(1-2^{k-1}\epsilon)}\bar{B}_t\right)= \sum_{k=1}^{L} \mathbb{P}\left(\max_i \frac{M^i_k}{b_i} > \frac{1}{n}{ 2^{-\frac{k}{2}}\epsilon^{\frac{1}{2}}}\right) \leq  m L\exp{\left(\frac{-\epsilon^2}{6 \gamma }\right)}.
\end{align}

\begin{align}\notag
 \forall k \in [L]:\quad &\mathbb{P}\left(N_k> \frac{1}{n}{P^* 2^{-\frac{k}{2}}\epsilon^{\frac{1}{2}}}\right) \leq \inf_{r \geq 0}{E\left[\exp{\left(-r S_{n(1-2^{k-1 }\epsilon)}\right)}\right]}{\exp{\left(-\frac{1}{n}{r P^*2^{-\frac{k}{2}}\epsilon^{\frac{1}{2}}}\right)}} \leq \exp{\left(\frac{-\epsilon^2}{4 \gamma }\right)}\\\label{n-bound}
%\Rightarrow \quad &\mathbb{P}\left(\bigcup\limits_{s=1}^{n(1-\epsilon)} \bar{C}_s\right)  \leq  L \exp{\left(\frac{-\epsilon^2}{4\gamma }\right)}.\\ \label{n-bound}
\Rightarrow \quad &\mathbb{P}\left(\bigcup\limits_{t=1}^{n(1-\epsilon)} \bar{C}_t\right)  \leq \sum_{k=1}^{L} \mathbb{P}\left(\bigcup\limits_{n(1-2^{k}\epsilon)}^{ n(1-2^{k-1}\epsilon)}\bar{C}_t\right) =  \sum_{k=1}^{L} \mathbb{P}\left(N_k> \frac{1}{n}{P^* 2^{-\frac{k}{2}}\epsilon^{\frac{1}{2}}}\right) \leq L \exp{\left(\frac{-\epsilon^2}{4\gamma }\right)}. 
\end{align}

\begin{thm}\label{General-thm}
Suppose $\check{x}_t$, $1\leq t \leq n$ is given by algorithm \eqref{general-algorithm}.  If  $\e \leq \frac{1}{12}$ and $\gamma \leq \frac{\epsilon^2}{13 \log\left(m/\epsilon\right)}$, then:

\begin{align}
E\left[\sum_{t=1}^{n} {f_t(\check{x}_t)}+\psi\left(\sum_{t=1}^{n}{A_t \check{x}_t }\right)\right] \geq (1-12\e) P^*.
\end{align}

\end{thm}

\begin{remark}
The conclusion of the theorem can be restated as:

\begin{align}
C_{\mathcal{D}_2} \geq 1 - 12 \e,
\end{align}
where $\mathcal{H}$ is the set of all $\left(\left(f_1,A_1\right), \left(f_2,A_2\right),\ldots, \left(f_n,A_n\right)\right) \in \left(\mathcal{G} \times \mathbb{R}_{+}^{m\times k}\right)^n$ that satisfy~\eqref{gamma} with $\gamma = \frac{\epsilon^2}{13 \log\left(m/\epsilon\right)}$, and with $P^* > 0$.
\end{remark}
\begin{proof}

Here we state the main steps of the proof. The details are presented in appendix A. Let $T= n(1-2\epsilon)$. For all $t \in \{n\e+1,n\e+2, \ldots,T\}$, define:

\begin{align*}
&\quad F_t = {C_t}\cap {B_{t}}\cap{\left\{P^* \leq q_t \leq P^*/\alpha_t\right\}},\\
\forall i \in [m]:&\quad\phi_i ^t=  \exp{\left( \sum_{s=n\e+1}^{t} \frac{\nu}{b_i}\left(A_{\sigma(s)}x_s\right)_i -  \beta_s  \right)},\\
& \quad \chi ^t= \exp{\left(  \sum_{s=n\e+1}^{t}-\frac{\nu' }{q_{s}} f_{\sigma(s)}(x_s) + \beta'_s\right)},\\
&\quad  \Phi^t = \left(\sum_{i=1}^{m} {\phi_i^t}+ m \chi^t\right)\prod_{s=n\e+1}^{t} 1_{F_s}.
\end{align*}

For all $i \in [m]$, define $\phi_i^{n\e} = 1$, $\chi^{n\e} = 1 $, and $\Phi^{n\e} = \sum_{i=1}^{m} \phi_i^{n\e}+m\chi^{n\e} = 2 m$. Let $\mathcal{A}_{n\e} = \{\varnothing, \Pi \}$, and for all $t \in \{n\e+1,n\e+2, \ldots,T\}$, let $\mathcal{A}_t$ be the sigma algebra generated by $\sigma(1),\ldots, \sigma(t)$.  We show in appendix A that $\left\{ \Phi_t, {\cal A}_t\right\}$, $t \in \{n\e, n\e+1, \ldots, T\}$, is a super-martingale, which yields:

\begin{align}
E\left[{\Phi^{T}} \right] \leq  E\left[{\Phi^{n\e}} \right]  = 2 m.
\end{align}

Let $B =\left\{\frac{1}{P^*}{\sum_{s=n\epsilon+1}^{T}f_{\sigma(s)}(x_s)} < \left(1-11\epsilon\right)\right\} $, $C =\left\{\max_i \frac{1}{b_i}{\sum_{s=n\epsilon+1}^{T}\left(A_{\sigma(s)}x_s\right)_i}> 1\right\} $, and $F =  \cap_{s=1}^{T} {F}_s$.  Using Markov's inequality, we can derive:

\def \chiB {\exp{\left(-\frac{\e^2}{7\gamma}\right)}} 
\def \phiB {\exp{\left(-\frac{\e^2}{4\gamma } \right)}}

\begin{align*}
 \mathbb{P}\left( B \cap F\right) & \leq  E\left[ \chi^{T}1_F\right]\chiB,\\
 \mathbb{P}\left( C \cap F\right) & \leq  E\left[\sum_{i=1}^{m} \phi_i^{T}1_F\right]\phiB.
\end{align*}

Combining the previous bounds with \eqref{Ph-bound}, \eqref{Dh-bound}, \eqref{m-bound} and \eqref{n-bound}, we get: 

\begin{align*}
\mathbb{P}\left( B \cup C\right)\leq \mathbb{P}\left( B \cap F\right) + \mathbb{P}\left( C \cap F\right) + \mathbb{P}\left(\bar{F}\right)& \leq  E\left[ \chi^{T}1_F\right]\chiB +  E\left[\sum_{i=1}^{m} \phi_i^{T}1_F\right]\phiB\\
%---------------------------------------------------------------------------------------------------------------------------------------------
 &+ 2 L  \exp{\left(\frac{-\epsilon^2}{4\gamma}\right)}+(2m+1) L \exp{\left(\frac{-\epsilon^2}{6\gamma}\right)}\\
%---------------------------------------------------------------------------------------------------------------------------------------------
&\leq E\left[ \Phi^{T}\right]\chiB+ 2 L  \exp{\left(\frac{-\epsilon^2}{4\gamma}\right)}+(2m+1) L \exp{\left(\frac{-\epsilon^2}{6\gamma}\right)}\\
%---------------------------------------------------------------------------------------------------------------------------------------------
&\leq 2 m\chiB +  2 L \exp{\left(\frac{-\epsilon^2}{4\gamma}\right)}+(2m+1) L\exp{\left(\frac{-\epsilon^2}{6\gamma}\right)}\leq \e,
%---------------------------------------------------------------------------------------------------------------------------------------------
\end{align*}

\noindent where in the last line we used the fact that $L = \log_2(\frac{1}{\e})\leq \log_2(e) \e^{-\frac{1}{e}}$ alongside the assumption that $\e \leq \frac{1}{12}$. As a result, we get:

\begin{align*}
%---------------------------------------------------------------------------------------------------------------------------------------------
E\left[\sum_{t=1}^{n} {f_t(\check{x}_t)}+\psi\left(\sum_{t=1}^{n}{A_t \check{x}_t }\right)\right] &\geq E\left[\left.\sum_{t=n\e+1}^{T} {f_t(x_t)}+\psi\left(\sum_{t=n\e+1}^{T}{A_t x_t }\right)\right\lvert\bar{B}\cap \bar{C}\right] \mathbb{P}\left(\bar{B}\cap \bar{C}\right)\\
&\geq (1-11\e)(1-\e)P^* \geq (1-12\e) P^*.
%---------------------------------------------------------------------------------------------------------------------------------------------
\end{align*}
\end{proof}

\subsection{Connection with subgradient methods}\label{sec::connection}
In this section, we discuss how the updates of $(y_t,y'_t)$ in algorithm~\eqref{general-algorithm} are equivalent to online exponentiated subgradient steps and point to the similarities and differences between ESA, DLA and KRTV. Recall that a pair $\left(\left(x^*_1,x^*_2,\ldots,x^*_n\right),y^*\right)$ is an optimal primal-dual pair for~\eqref{vector_general} if and only if:

\begin{align} 
&  y^* \in  -\partial \psi\left(\sum_{t=1}^{n}{A_t x_t^* } \right),\\  \notag
\\ 
&\forall t \in [n]: \quad A_t^\top y^* \in  \partial{f_t}\left(x_t^*\right). \label{cond_g_2_2} 
\end{align}

Note that \eqref{cond_g_2_2} is decoupled in $t$. Therefore, to assign $x_t$, an online algorithm can find $y_t$, an estimate of $y^*$, and then choose $x_t$ such that:

\begin{align} \label{update_rule}
 x_t \in  \partial{f^*_{\s(t)}}\left( A_{\s(t)}^{\top} y_t\right)  \leftrightarrow  A_{\s(t)}^{\top} y_t \in  \partial{f_{\s(t)}}\left(x_t\right).
\end{align}

ESA and DLA are two examples of online algorithms that choose $x_t$ according to \eqref{update_rule}. KRTV can also be transformed into an algorithm that chooses $x_t$ according to \eqref{update_rule} (Both DLA and KRTV are proposed for online LP. Here we consider extensions of those algorithms to \eqref{vector_general}). However, these algorithms update $y_t$ differently. 
When $t = n \e 2^h$ for some $h \in\{0,1,\ldots,L-1\}$, DLA sets  $y_t$ to be:

% the solution to the following dual problem{\footnote{ Note that in the general case to ensure that such a solution exists, one can impose a simple condition on the domain of functions (see chapter 11 \cite{rockafellar1998variational}).  }}:

\begin{align}\label{dual_estimate}
y_t\in\argmin_{y \in \mathbb{R}^m} {\quad \frac{1}{2^{h} \epsilon(1-2^{-\frac{h}{2}}\e^{\frac{1}{2}})} \sum_{s=1}^{t} {-f^*_{\sigma(s)}\left(A_{\sigma(s)}^{\top} y\right)} - \psi^*\left(-y \right)},\\ \notag
\end{align}

\noindent then for all $t \in \{n\e 2^h+1,n\e 2^h+2, \ldots, n\e2^{h+1}\}$ sets $y_t = y_{n \e 2^h}$. Although the dual variables do not explicitly appear in the description of KRTV in \cite{kesselheimRTV13}, this algorithm can be represented as an extreme case that updates $y_t$ by solving~\eqref{dual_estimate} (without $(1-2^{-\frac{h}{2}}\e^{\frac{1}{2}})$) for every $t \in [n]$. ESA also updates $y_t$  for every $t \in [n]$. However, the updates are very simple and computationally very cheap. Consider the following problem which is dual to \eqref{vector_general_equiv}:

%\begin{align}\tag{$\mathcal{D}2$}
%D^* = \inf_{\begin{array}{c}y \in \mathbb{R}^m,\, y'\in \mathbb{R} \end{array}}&{  \quad  P^* - \psi^*\left(-y \right) -\hat{\psi}^*\left(-y' \right) + \sum_{t=1}^{n} {\sup_{x \in \mathbb{R}^k}{y' f_t(x)- y^\top A_t x}}}.\\ \notag
% = \inf_{\begin{array}{c}y \in \mathbb{R}^m_+,\, y'\in \mathbb{R}_+ \end{array}}&{  \quad   P^* + b^T y - P^* y' + \sum_{t=1}^{n}{\sup_{x \in \mathbb{R}^k}{y' f_t(x)- y^\top A_t x}}}\\ \notag
% = \inf_{\begin{array}{c}y \in \mathbb{R}^m_+,\, y'\in \mathbb{R}_+ \end{array}}&{  \quad   P^* + \mathbf{1}^{\top} y / m -  y' + \sum_{t=1}^{n}{H_{t}(y,y')}}\\ \notag
%%\inf_{\begin{array}{c}y \in \mathbb{R}^m,\, y'\in (0,1]\end{array}}&{  \quad \sum_{t=1}^{n} {-y' f^*_{t}(A_{t}^{\top} \frac{y}{y'})} - \psi^*\left(-y \right)}.\\ \notag
%\end{align}
  
\begin{align}\tag{$\mathcal{D}2$}\label{vector_general_equiv_dual}
\underset{y \in \mathbb{R}^m,\, y'\in \mathbb{R}}{\mbox{minimize}}&{\quad   - \psi^*\left(-y \right) -\hat{\psi}^*\left(-y' \right) + \sum_{t=1}^{n} {\sup_{x \in \mathbb{R}^k}{y' f_t(x)- y^\top A_t x}}}. \\ \notag
\end{align}

Note that $\psi^*\left(u\right) = b^T u$ and $\hat{\psi}^*\left(v\right) = - P^* v$ for all $u \in \mathbb{R}_{+}^m$ and $v \in \mathbb{R}_{+}$. Using the change of variables $y_i \rightarrow y_i/b_i$ and $y \rightarrow y/ P^*$, we can rewrite \eqref{vector_general_equiv_dual} as:

\begin{align}
\underset{y \in \mathbb{R}_{+}^m,\, y'\in \mathbb{R}_{+}}{\mbox{minimize}}&{\quad  \mathbf{1}^{\top} y -  y' + \sum_{t=1}^{n}{H_{t}(y,y')}},\\ \notag
\end{align}

\noindent where $H_{t}(y,y') = \sup_{x \in \mathbb{R}^k}{\frac{y'}{P^*} f_t(x)- \sum_{i=1}^{m}\frac{y_i}{b_i} \left(A_t x\right)_i}$. By replacing $P^*$ with $q_t$, we can define:
  $$\tilde{H}_{\s(t)}(y,y') = \sup_{x \in \mathbb{R}^k}{\frac{y'}{q_t} f_{\s(t)}(x)- \sum_{i=1}^{m}\frac{y_i}{b_i} \left(A_{\s(t)} x\right)_i}.$$
  
Let $x_t$, $n\e< t \leq n$ be given by ESA. Since for all $t \in \{n\e+1,n\e+2,\ldots,n\}$, $x_t \in \partial f_{\s(t)}^*(A_{\s(t)}^{\top} {y''_t})$, we get:

%$H_{\s(t)}(y,y') = -y' f^*_{\s(t)}(A_{\s(t)}^{\top} \frac{y}{y'})$.

\begin{align}\label{subg}
\left[\begin{array}{ccccc}\frac{-1}{b_1} \left(A_{\s(t)} x_t\right)_1,&\frac{-1}{b_2} \left(A_{\s(t)} x_t\right)_2,& \ldots&\frac{-1}{b_m} \left(A_{\s(t)} x_t\right)_m,&  \frac{1}{q_t}f_{\s(t)}(x_t)\end{array}\right]^{\top} \in \partial \tilde{H}_{\s(t)}(y_t,y'_t).
\end{align}

%ESA sets $y_1 = \mathbf{1}$ and $y' = 1$.When ESA assigns $x_t$, it updates $y_t$ via an online exponentiated in the negative direction of a subgradient of  $\partial g(y,y')$:

For all $t \in \{ n\e+1, \ldots, n - 1\}$, ESA updates $(y_t,y'_t)$ via an online exponentiated step toward the negative direction of the subgradient of $\tilde{H}_{\s(t)}(y_t,y'_t)$ given in \eqref{subg}:

\begin{align*}
\forall i \in [m]:\quad  \left(y_{t+1}\right)_i &= y_t \exp{\left(\frac{\nu}{b_i}\left( A_{\s(t)} x_t\right)_i -  \beta_{t+1}\right)},\\
\quad y'_{t+1} &=  y'_t \exp{\left(-\frac{\nu'}{q_t}f_{\s(t)}(x_t)+ \beta'_{t+1}\right)},
\end{align*}

\noindent where ${\nu}$ and ${\nu'}$ are the step sizes while $\exp{\left(- \beta_{t+1}\right)}$ and $\exp{\left( \beta'_{t+1}\right)}$ are the normalization terms.
\section{Numerical experiment}

In this section, we examine the performance of ESA and compare it with that of four other online LP algorithms proposed in the literature. The algorithms are \emph{one-time learning algorithm (OLA)}, DLA, KRTV, and KRTV5. OLA is a simpler version of DLA introduced in \cite{agrawal2009dynamic}. OLA computes the dual variable at $t = n\e$ and uses the same dual variable to choose $x_t$ for all $t \in \{n\e+1,n\e+2,\ldots,n\}$. As it was discussed in section~\ref{sec::connection}, KRTV can be viewed as an algorithm that chooses $x_t$ according to \eqref{update_rule} and updates the dual variable at each step. This point of view motivated us to consider a variant of KRTV that updates the dual variable after every fixed number of steps. In KRTV5, the dual variable is updated every $5$ steps. 

Consider an online LP problem and let $M = \{\{\left(c_t, A_t\right)\lvert 1\leq t\leq n\}\}$, where the double bracket notation denotes a multiset. we generate $M$ based on the construction proposed in \cite{agrawal2009dynamic} for deriving the necessary condition on $\gamma$. In this scheme of construction, $m = 2^d$ for some integer $d$. $b = c \mathbf{1}$ for some $c \in \mathbb{R}_{++}$. For all $l \in [m]$, let $l-1 = (a_{l,1},a_{l,2},\ldots,a_{l,d})_2 $. For all $i \in [d]$, define $v_i = [a_{1,i},a_{2,i},\ldots,a_{m,i}]^T$ and $w_i = \mathbf{1} - v_i$. Let $j_1, j_2, \ldots, j_m$ be $m$ independent samples from $Binomial(\left\lceil\frac{2 c}{d}\right\rceil,\frac{1}{2})$. $M$ consists of:

\begin{enumerate}
\item $(4,v_i)$ with multiplicity $\left\lceil\frac{c}{d}\right\rceil$ for all $ i \in [d]$
\item $(3,w_i)$ with multiplicity $j_i$ for all $ i \in [d]$
\item $(2,w_i)$ with multiplicity $\left\lceil\frac{1}{2}\sqrt{\frac{c}{d}}\right\rceil$ for all $ i \in [d]$
\item $(1,w_i)$ with multiplicity $\left\lceil\frac{2 c}{d}\right\rceil - j_i$ for all $ i \in [d]$
\end{enumerate}

Table~\ref{Table1} compares the competitive ratio and the runtime of all the algorithms when $d = 3$ and $c \in \{1200,300,60,30\}$. 
%For each value of $\e$, we have considered two problems with $c = \frac{d}{\e^2}$ and $c = \frac{d}{\e}$.
The competitive ratio of the algorithms is estimated by averaging the performance for $500$ permutations. The competitive ratios (CR) reported in Table~\ref{Table1} are the average of estimated competitive ratios for 10 independently generated instances of $M$. The runtime reports the average time it takes an algorithm to provide a solution for one permutation of a single instance of $M$. For this numerical experiment, we used Matlab R2014a and the linear programming solver of Matlab optimization toolbox. The runtimes are collected from a Linux machine with an Intel Core  i7.4770 3.40GHz CPU and 8GB of RAM. In ESA, $\gamma$ is set to be equal to $\frac{1}{c}$. Note that $\e$ is as an input parameter for ESA, OLA, and DLA. For each $\e \in \{0.1,0.05\}$, we have considered two problems with $c = \frac{d}{\e^2}$ and $c = \frac{d}{\e}$. When $c = \frac{d}{\e^2}$, the bid-to-budget ratio matches the necessary condition for achieving a competitive ratio of $1 - O(\e)$. As the data presented in Table~\ref{Table1} suggest, ESA, OLA, and DLA are multiple orders of magnitude faster than KRTV and  KRTV5. When $c \in \{300,60,30\}$, ESA with $\e = 0.05$ achieves a competitive ratio that is higher than or equal to the rest of the algorithms. Note that ESA, OLA, and DLA choose $x_t = 0$ for $t \in [n\e]$. Despite this fact, when $\e = 0.05$ and $c = \frac{d}{0.05^2} = 1200$, the difference between the competitive ratio of ESA and KRTV is less than $\e$.

\begin{table}
\caption{Comparison between OLA, DLA, ESA, KRTV, KRTV5 when $c \in \{1200,300,60,30\}$}\label{Table1}
\begin{center}
\begin{tabular}{ll  l l   l l  l l  l l}
\toprule
%\hline
 & & \multicolumn{2}{c}{ $n = 3630, c =  1200$}&\multicolumn{2}{c}{$n = 915, c = 300$}& \multicolumn{2}{c}{ $n = 189, c = 60$}& \multicolumn{2}{c}{ $n = 96, c = 30$}  \\
\cmidrule(r){3-4}\cmidrule(r){5-6}\cmidrule(r){7-8}\cmidrule(r){9-10}
%\hline
             Algorithm   &   $\e$    & CR      &time (s)    &CR  & time (s)  &CR &time (s)    & CR  & time (s)\\
%\midrule
\cmidrule(r){1-2}\cmidrule(r){3-4}\cmidrule(r){5-6}\cmidrule(r){7-8}\cmidrule(r){9-10}
%\hline
\multirow{2}{*}{OLA}   & 0.1             &  0.88 &    0.060    & 0.79& 0.020	 &  0.72&  0.012   & 0.65  & 0.011 \\
                                      & 0.05            &  0.91&    0.064    & 0.83&	0.020  & 0.69&  0.022 & 0.67  &   0.011\\
                     %                 &                    &         &                  &        &             &          &           &          &             \\ 
\cmidrule(r){1-2}\cmidrule(r){3-4}\cmidrule(r){5-6}\cmidrule(r){7-8}\cmidrule(r){9-10}
%\hline
\multirow{2}{*}{DLA}  & 0.1               &  0.89&      0.112   &  0.74&  0.05&  0.79&    0.036    &0.76    &  0.033 \\
			     & 0.05             &  0.94&      0.148 &  0.92 &  0.065  &  0.84& 0.081 & 0.80   &    0.043\\
                     %                &                     &          &                &          &           &          &           &            &       \\
\cmidrule(r){1-2}\cmidrule(r){3-4}\cmidrule(r){5-6}\cmidrule(r){7-8}\cmidrule(r){9-10}
%\hline
\multirow{2}{*}{ESA}  & 0.1              &  0.86 &   0.173   &  0.87 & 0.068  &  0.90 & 0.041  & 0.86   &   0.036\\
                                     & 0.05            &  0.93 &   0.205    &0.95   &  0.084 &  0.92& 0.093 &  0.87  &     0.046  \\
		 %              &                     &          &                 &          &           &           &           &          &               \\ 
\cmidrule(r){1-2}\cmidrule(r){3-4}\cmidrule(r){5-6}\cmidrule(r){7-8}\cmidrule(r){9-10}
%\hline
KRTV                              &n/a              &  0.97&132.636    &0.93  &17.485 &  0.88&4.272 &   0.84 &    1.125\\ 
\cmidrule(r){1-2}\cmidrule(r){3-4}\cmidrule(r){5-6}\cmidrule(r){7-8}\cmidrule(r){9-10}
%\hline
KRTV5                          &n/a                &  0.98&  27.828     & 0.94 &3.493  & 0.90& 0.831 &   0.87 &   0.221\\
\bottomrule
%\hline
%\toprule
%& \multicolumn{3}{c}{ $n = 3630, c = \frac{d}{\epsilon^2} = 1200$}& \multicolumn{3}{c}{ $n = 189, c = \frac{d}{\epsilon} = 60$} \\
%\cmidrule(r){2-4}\cmidrule(r){5-7}
%Algorithm    & $\e$ &CR &time (s)    &$\e$& CR  & time (s)\\
%%\midrule
%\cmidrule(r){1-4}\cmidrule(r){5-7}
%\multirow{3}{*}{OLA}& 0.1  &             &              &    0.1 &  0.72&            \\
%                                   & 0.05 &    0.91&    0.064 & 0.05   & 0.69& 0.022\\
%		               &           &          &           	 &          &           &      \\ 
%\cmidrule(r){1-4}\cmidrule(r){5-7}
%\multirow{3}{*}{DLA}             \\
%                                   \\
%		                   \\ 
%\cmidrule(r){1-4}\cmidrule(r){5-7}
%\multirow{3}{*}{ESA}         \\
%                                   \\
%		               &           &          &           	 &          &           &      \\ 
%\cmidrule(r){1-4}\cmidrule(r){5-7}
%KRTV \\ 
%\cmidrule(r){1-4}\cmidrule(r){5-7}
%KRTV5  &&   0.98&  27.828& & 0.90& 0.831\\ 
%\bottomrule
\end{tabular}
\end{center}
\end{table}

%\section{Discussion and future work}
%
%Our next goal is to generalize the analysis of ESA to the case where $\psi$ is a fairly general concave function.  Another question that remains to be answered is that whether the condition on $\gamma$ can be improved from $\gamma = O\left(\frac{\e^2}{\log\left(\frac{m}{\e}\right)}\right)$ to $\gamma = O\left(\frac{\e^2}{\log{m}}\right)$.
%
%%We have considered a possible extension of ESA to general concave $psi$, what remain is the analysis.
%As it was noted in the numerical section, DLA's performance is at the same level as KRTV and ESA. However, there is a sizable gap between the theoretical  analysis of DLA and its performance. We believe an analysis for that uses more advanced probabilistic tool rather than a simple union bound on the space of dual variable may close this gap.

%\input{Appendix.tex}
\section*{Appendix A: Details for the proof of Theorem \ref{General-thm}}\label{AppA} 

First, we show that $\left\{ \Phi_t, {\cal A}_t\right\}$, $t \in \{n\e, n\e+1, \ldots, T\}$ is a super-martingale.

\begin{align*}
E\left[\left. {\Phi^{t}}\right \lvert {\cal A}_{t-1} \right] &= E\left[\left. \left(\sum_{i=1}^{m} {\phi_i^t}+m \chi^t\right)\prod_{s=n\e+1}^{t} 1_{F_s}\right \lvert {\cal A}_{t-1} \right]\\
%-----------------------------------------------------------------------------------------------------------------------------------------------------------------------------------------------------------------
&=E\left[\left. \left(\sum_{i=1}^{m} {\phi_i^{t-1}}\exp{\left(\frac{\nu}{b_i} \left(A_{\sigma(t)}x_{t}\right)_i - \beta_t \right)}+m \chi^{t-1} \exp{\left(-\frac{\nu'}{q_t} f_{\sigma(t)}\left(x_{t}\right)+ \beta'_t \right)}\right)  \prod_{s=n\e+1}^{t} 1_{F_s}\right \lvert {\cal A}_{t-1} \right]\\
%-----------------------------------------------------------------------------------------------------------------------------------------------------------------------------------------------------------------
&\leq E\left[\left. \left(\sum_{i=1}^{m} {\phi_i^{t-1}}\left(1+\frac{ \e }{\gamma b_i}\left(A_{\sigma(t)}x_{t}\right)_i \right)\exp{\left( - \beta_t \right)}+  m\chi^{t-1} \left(1-\frac{\e}{\gamma q_t} f_{\sigma(t)}\left(x_{t}\right) \right) \exp{\left( \beta'_t \right)}\right)  \prod_{s=n\e+1}^{t} 1_{F_s}\right \lvert {\cal A}_{t-1} \right]\\
%-----------------------------------------------------------------------------------------------------------------------------------------------------------------------------------------------------------------
&\leq E\left[\left. \left(\sum_{i=1}^{m} {\phi_i^{t-1}}\left(1+\frac{ \e}{\gamma b_i}  \left(A_{\sigma(t)}x^*_{\sigma(t)}\right)_i\right)\exp{\left( - \beta_t \right)}+ m\chi^{t-1}  \left(1-\frac{\e }{\gamma q_t}f_{\sigma(t)}\left(x^*_{\sigma(t)}\right) \right) \exp{\left( \beta'_t \right)}\right)   \prod_{s=n\e+1}^{t}1_{F_s}\right \lvert {\cal A}_{t-1} \right]\\
%-----------------------------------------------------------------------------------------------------------------------------------------------------------------------------------------------------------------
 &=\sum_{i=1}^{m} {\phi_i^{t-1}} E\left[\left. \left(1+  \frac{\e \left(A_{\sigma(t)}x^*_{\s(t)}\right)_i }{\gamma b_i} \right)\exp{\left(-\beta_t \right)}  \right \lvert {\cal A}_{t-1} \right] \prod_{s=n\e+1}^{t} 1_{F_s}\\
&\phantom{111}+ m \chi^{t-1} E\left[\left. \left(1 - \frac{\e f_{\sigma(t)}\left(x^*_{\sigma(t)}\right)}{\gamma q_t   }\right)\exp{\left(\beta'_t \right)}  \right \lvert {\cal A}_{t-1} \right]\prod_{s=n\e+1}^{t} 1_{F_s}\\
%----------------------------------------------------------------------------------------
 &=\sum_{i=1}^{m} {\phi_i^{t-1}}  \left(1+ \frac{\e}{\gamma b_i}{\left(R_t^i +  b_i /n\right)} \right)\exp{\left(-\beta_t \right)} \prod_{s=n\e+1}^{t} 1_{F_s}  +m\chi^{t-1}  \left(1 - \frac{\e }{\gamma q_t}{\left(S_t +  P^*/n\right)}\right)\exp{\left(\beta'_t \right)}\prod_{s=n\e+1}^{t} 1_{F_s}  \\
%--------------------------------------------------------------------------------------------------
&\leq \left(\phi^{t-1} \exp{  \left({\beta_t} - {\beta_t}\right)}+m\chi^{t-1} \exp{ \left(-{\beta'_t}  + \beta'_t \right)}\right)  \prod_{s=n\e+1}^{t-1} 1_{F_s} = \Phi^{t-1}.
%&& \text{by \eqref{Bernstein}}
\end{align*}

The first inequality follows from \eqref{el3}; the second inequality follows from the fact that

% $$x_t \in \argmin_{z}{\sum_{i=1}^{m}{\frac{1}{b_i}  \exp{\left(\frac{\nu}{b_i} \sum_{s=1}^{t-1}\left(A_{\sigma\left(s\right)}x_s\right)_i  + \frac{\nu'}{q_s} f_{\sigma(s)}\left(x_s\right)\right)} \left(A_{\sigma\left(t\right)}z\right)_i } -\frac{1}{q_s} f_{\sigma(t)}\left(z\right) },$$

$$x_t \in \argmin_{z}{ {y''_t}^{\top} A_{\sigma\left(t\right)}z  - f_{\sigma(t)}\left(z\right) };$$

\noindent the last inequality follows from the definition of $F_t$. Using the definition of $\phi_i^t$ and $\chi^t$ alongside Markov's inequality we can derive:

\def \chiB {\exp{\left(-\frac{\e^2}{7\gamma}\right)}} 
\def \phiB {\exp{\left(-\frac{\e^2}{4\gamma } \right)}}
\begin{align}\notag	
\mathbb{P}\left(\left\{\frac{1}{P^*}{\sum_{s=1}^{T} f_{\sigma(s)}(x_s)} < \left(1 - 3\e \right)(1-8\e)\right\} \cap F\right)&\leq \mathbb{P}\left(\left\{\sum_{s=n\epsilon+1}^{T} \frac{f_{\sigma(s)}(x_s)}{q_s} < \left(1 - 3\e \right)(1-8\e) \right\} \cap F\right)\\ \notag	
%-----------------------------------------------------------------------------------------------------------------------------------------------------------------------------------------------------------------
&=  \mathbb{P}\left(\left\{\chi^{T}> \exp{ \left( \sum_{s=n\epsilon+1}^{T} \beta'_s -\nu' \left(1 - 3\e \right)(1-8\e)  \right)}\right\}\cap F\right) \\\notag	
%-----------------------------------------------------------------------------------------------------------------------------------------------------------------------------------------------------------------
&\leq  \mathbb{P}\left(\chi^{T}1_F> \exp{ \left( \sum_{s=n\epsilon+1}^{T} \beta'_s -\nu' \left(1 - 3\e \right)(1-8\e)  \right)}\right) \\\notag	
%-----------------------------------------------------------------------------------------------------------------------------------------------------------------------------------------------------------------
%&\leq E\left[\chi^{T}1_F\right] \exp{\left(\nu' \left(1 - 3\e \right)(1-8\e) - \sum_{s=n\epsilon+1}^{T} \beta'_s\right)}\\\notag	
%-----------------------------------------------------------------------------------------------------------------------------------------------------------------------------------------------------------------
&\leq E\left[\chi^{T}1_F\right] \exp{\left(\frac{1 - 3\e}{\gamma}  \left(\gamma \nu' (1-8\e) - \e (1-7\e)\right)\right)}\\\notag	
%-----------------------------------------------------------------------------------------------------------------------------------------------------------------------------------------------------------------
&\leq E\left[\chi^{T}1_F\right] \exp{\left(\frac{(1 - 3\e)(1-7\e)}{\gamma}  \left(\left(\frac{\e}{1-7\e}-1\right)\log{\left(1-\e\right)} - \e\right)\right)}\\ \label{P(B)}	
%-----------------------------------------------------------------------------------------------------------------------------------------------------------------------------------------------------------------
&\leq E\left[\chi^{T}1_F\right] \exp{\left(-\frac{ (1-3\e) (1-7\e) \e^2}{2\gamma}\right)}\leq  E\left[\chi^{T}1_F\right] \chiB,
\end{align}

\noindent where the last line follows from:

\begin{align}
&\forall u \in [0,1]:  \quad u + (1-  u)\log(1- u) \geq \frac{1}{2} u^2.
\end{align}

\begin{align}\notag
\mathbb{P}\left(\left\{\max_i  \frac{1}{b_i}{\sum_{s=1}^{T}\left(A_{\sigma(s)}x_s\right)_i} > \left(1-3\e \right)\left(1+ 3\epsilon\right)\right\} \cap F\right) 
&=  \mathbb{P}\left(\left\{\max_i \phi_i^{T} > \exp{\left( \nu \left(1-3 \e \right)\left(1+ 3\epsilon\right) - \sum_{s=n\e+1}^{T}   \beta_s  \right)}\right\} \cap F\right) \\\notag
%-----------------------------------------------------------------------------------------------------------------------------------------------------------------------------------------------------------------
&\leq  \mathbb{P}\left(\max_i \phi_i^{T}1_F > \exp{\left( \nu \left(1-3 \e \right)\left(1+ 3\epsilon\right) - \sum_{s=n\e+1}^{T}   \beta_s  \right)} \right) \\\notag
%-----------------------------------------------------------------------------------------------------------------------------------------------------------------------------------------------------------------
&\leq  \mathbb{P}\left(\sum_{i=1}^{m} \phi_i^{T} 1_F> \exp{\left(\nu\left(1-3\epsilon\right)\left(1+3\epsilon\right) - \sum_{s=n\e+1}^{T}{\beta_s}\right)}\right) \\\notag
%-----------------------------------------------------------------------------------------------------------------------------------------------------------------------------------------------------------------
%s&\leq \mathbb{P}\left(\sum_{i=1}^{m} \phi_i^{T} 1_F \geq \exp{\left(\frac{1-3\epsilon}{\gamma} \left( \gamma \nu\left(1+3\epsilon\right) -\e\left(1+2 \e\right) \right)\right)}\right) \\\notag
%-----------------------------------------------------------------------------------------------------------------------------------------------------------------------------------------------------------------
&\leq  E\left[\sum_{i=1}^{m} \phi_i^{T} 1_F \right] \exp{\left(\frac{1-3\epsilon}{\gamma} \left(\e(1+2\e) -  \gamma \nu\left(1+3\epsilon\right) \right)\right)}\\\notag
%-----------------------------------------------------------------------------------------------------------------------------------------------------------------------------------------------------------------
&\leq  E\left[\sum_{i=1}^{m} \phi_i^{T} 1_F\right] \exp{\left(\frac{(1-3\e)(1+2\e)}{\gamma} \left(\e -  \left(1+\frac{\e}{1+2\e}\right)\log{\left(1+\e\right)} \right)\right)}\\  \label{P(C)}
%-----------------------------------------------------------------------------------------------------------------------------------------------------------------------------------------------------------------
&\leq  E\left[\sum_{i=1}^{m} \phi_i^{T}1_F \right]   \exp{\left(-\frac{(1-3\e)(1+2\e)\e^2}{3 \gamma \left(1+\e/3\right)} \right)} \leq  E\left[\sum_{i=1}^{m} \phi_i^{T}1_F \right] \phiB,
%-----------------------------------------------------------------------------------------------------------------------------------------------------------------------------------------------------------------
\end{align}

\noindent where the last line follows from \eqref{el1}.

%\input{S4.tex}
%\input{S1.tex}
%\input{S2.tex}
%Bibliography
\bibliographystyle{alpha}	
\bibliography{OnRef}

\newcommand{\etalchar}[1]{$^{#1}$}
\begin{thebibliography}{FMMM09}

\bibitem[AD14]{agrawal2014fast}
Shipra Agrawal and Nikhil~R Devanur.
\newblock Fast algorithms for online stochastic convex programming.
\newblock {\em arXiv preprint arXiv:1410.7596}, 2014.

\bibitem[AWY09]{agrawal2009dynamic}
Shipra Agrawal, Zizhuo Wang, and Yinyu Ye.
\newblock A dynamic near-optimal algorithm for online linear programming.
\newblock {\em arXiv preprint arXiv:0911.2974}, 2009.

\bibitem[BIKK07]{babaioff2007knapsack}
Moshe Babaioff, Nicole Immorlica, David Kempe, and Robert Kleinberg.
\newblock A knapsack secretary problem with applications.
\newblock In {\em Approximation, Randomization, and Combinatorial Optimization.
  Algorithms and Techniques}, pages 16--28. Springer, 2007.

\bibitem[BIKK08]{babaioff2008online}
Moshe Babaioff, Nicole Immorlica, David Kempe, and Robert Kleinberg.
\newblock Online auctions and generalized secretary problems.
\newblock {\em SIGecom Exch.}, 7(2):7:1--7:11, June 2008.

\bibitem[BJN07]{buchbinder2007online}
Niv Buchbinder, Kamal Jain, and Joseph~Seffi Naor.
\newblock Online primal-dual algorithms for maximizing ad-auctions revenue.
\newblock In {\em Algorithms--ESA 2007}, pages 253--264. Springer, 2007.

\bibitem[BLM13]{boucheron2013concentration}
St{\'e}phane Boucheron, G{\'a}bor Lugosi, and Pascal Massart.
\newblock {\em Concentration inequalities: A nonasymptotic theory of
  independence}.
\newblock Oxford University Press, 2013.

\bibitem[DH09]{devanur2009adwords}
Nikhil~R. Devanur and Thomas~P. Hayes.
\newblock The adwords problem: Online keyword matching with budgeted bidders
  under random permutations.
\newblock In {\em Proceedings of the 10th ACM Conference on Electronic
  Commerce}, EC '09, pages 71--78, New York, NY, USA, 2009. ACM.

\bibitem[DJ12]{devanur2012online}
Nikhil~R Devanur and Kamal Jain.
\newblock Online matching with concave returns.
\newblock In {\em Proceedings of the forty-fourth annual ACM symposium on
  Theory of computing}, pages 137--144. ACM, 2012.

\bibitem[DJSW11]{devanur2011near}
Nikhil~R Devanur, Kamal Jain, Balasubramanian Sivan, and Christopher~A Wilkens.
\newblock Near optimal online algorithms and fast approximation algorithms for
  resource allocation problems.
\newblock In {\em Proceedings of the 12th ACM conference on Electronic
  commerce}, pages 29--38. ACM, 2011.

\bibitem[FHK{\etalchar{+}}10]{feldman2010online}
Jon Feldman, Monika Henzinger, Nitish Korula, Vahab~S. Mirrokni, and Cliff
  Stein.
\newblock Online stochastic packing applied to display ad allocation.
\newblock In {\em Proceedings of the 18th Annual European Conference on
  Algorithms: Part I}, ESA'10, pages 182--194, Berlin, Heidelberg, 2010.
  Springer-Verlag.

\bibitem[FMMM09]{feldman2009online2}
Jon Feldman, Aranyak Mehta, Vahab Mirrokni, and S~Muthukrishnan.
\newblock Online stochastic matching: Beating 1-1/e.
\newblock In {\em Foundations of Computer Science, 2009. FOCS'09. 50th Annual
  IEEE Symposium on}, pages 117--126. IEEE, 2009.

\bibitem[GM14]{gupta2014experts}
Anupam Gupta and Marco Molinaro.
\newblock How the experts algorithm can help solve lps online.
\newblock {\em arXiv preprint arXiv:1407.5298}, 2014.

\bibitem[JL12]{jaillet2012near}
Patrick Jaillet and Xin Lu.
\newblock Near-optimal online algorithms for dynamic resource allocation
  problems.
\newblock {\em arXiv preprint arXiv:1208.2596}, 2012.

\bibitem[JL13]{jaillet2013online}
Patrick Jaillet and Xin Lu.
\newblock Online stochastic matching: New algorithms with better bounds.
\newblock {\em Mathematics of Operations Research}, 2013.

\bibitem[Kle05]{kleinberg2005multiple}
Robert Kleinberg.
\newblock A multiple-choice secretary algorithm with applications to online
  auctions.
\newblock In {\em Proceedings of the sixteenth annual ACM-SIAM symposium on
  Discrete algorithms}, pages 630--631. Society for Industrial and Applied
  Mathematics, 2005.

\bibitem[KMT11]{karande2011online}
Chinmay Karande, Aranyak Mehta, and Pushkar Tripathi.
\newblock Online bipartite matching with unknown distributions.
\newblock In {\em Proceedings of the forty-third annual ACM symposium on Theory
  of computing}, pages 587--596. ACM, 2011.

\bibitem[KP09]{korula2009algorithms}
Nitish Korula and Martin P{\'a}l.
\newblock Algorithms for secretary problems on graphs and hypergraphs.
\newblock In {\em Automata, Languages and Programming}, pages 508--520.
  Springer, 2009.

\bibitem[KRTV13]{kesselheim2013optimal}
Thomas Kesselheim, Klaus Radke, Andreas T{\"o}nnis, and Berthold V{\"o}cking.
\newblock An optimal online algorithm for weighted bipartite matching and
  extensions to combinatorial auctions.
\newblock In {\em Algorithms--ESA 2013}, pages 589--600. Springer, 2013.

\bibitem[KRTV14]{kesselheimRTV13}
Thomas Kesselheim, Klaus Radke, Andreas T{\"o}nnis, and Berthold V{\"o}cking.
\newblock Primal beats dual on online packing lps in the random-order model.
\newblock In {\em Proceedings of the 46th Annual ACM Symposium on Theory of
  Computing}, STOC '14, pages 303--312, New York, NY, USA, 2014. ACM.

\bibitem[KVV90]{karp1990optimal}
Richard~M Karp, Umesh~V Vazirani, and Vijay~V Vazirani.
\newblock An optimal algorithm for on-line bipartite matching.
\newblock In {\em Proceedings of the twenty-second annual ACM symposium on
  Theory of computing}, pages 352--358. ACM, 1990.

\bibitem[Meh13]{mehta13}
Aranyak Mehta.
\newblock Online matching and ad allocation.
\newblock {\em Foundations and Trends in Theoretical Computer Science},
  8(4):265--368, 2013.

\bibitem[MGS12]{manshadi2012online}
Vahideh~H Manshadi, Shayan~Oveis Gharan, and Amin Saberi.
\newblock Online stochastic matching: Online actions based on offline
  statistics.
\newblock {\em Mathematics of Operations Research}, 37(4):559--573, 2012.

\bibitem[MR13]{molinaro2013geometry}
Marco Molinaro and R~Ravi.
\newblock The geometry of online packing linear programs.
\newblock {\em Mathematics of Operations Research}, 39(1):46--59, 2013.

\bibitem[MSVV07]{mehta2007adwords}
Aranyak Mehta, Amin Saberi, Umesh Vazirani, and Vijay Vazirani.
\newblock Adwords and generalized online matching.
\newblock {\em Journal of the ACM (JACM)}, 54(5):22, 2007.

\bibitem[MY11]{mahdian2011online}
Mohammad Mahdian and Qiqi Yan.
\newblock Online bipartite matching with random arrivals: an approach based on
  strongly factor-revealing lps.
\newblock In {\em Proceedings of the forty-third annual ACM symposium on Theory
  of computing}, pages 597--606. ACM, 2011.

\bibitem[Ser74]{serfling1974probability}
Robert~J Serfling.
\newblock Probability inequalities for the sum in sampling without replacement.
\newblock {\em The Annals of Statistics}, pages 39--48, 1974.

\bibitem[Sho00]{shorack2000probability}
Galen~R Shorack.
\newblock {\em Probability for statisticians}.
\newblock Springer, 2000.

\end{thebibliography}

\end{document}